\documentclass[11pt]{amsart}
\usepackage{geometry}                % See geometry.pdf to learn the layout options. There are lots.
\usepackage{tikz}

\input xy \xyoption{all}
\usepackage[all,poly]{xy}

\usepackage{graphicx}
\usepackage{hyperref}
\usepackage{mathrsfs}
\usepackage{amsmath,amsfonts,amssymb,amsthm}
\usepackage{epstopdf}
\usepackage{tikz}
\usetikzlibrary{positioning}
\theoremstyle{plain}% default

\newtheorem{thm}{Theorem}[section]
\newtheorem{lem}[thm]{Lemma}
\newtheorem{prop}[thm]{Proposition}
\newtheorem{cor}[thm]{Corollary}

\theoremstyle{definition}
\newtheorem{defn}[thm]{Definition}

\newtheorem{exmp}[thm]{Example}

\theoremstyle{remark}
\newtheorem{rem}[thm]{Remark}

\newcommand{\C}{\mathbb{C}}

\newcommand{\cI}{\mathcal{I}}

\renewcommand{\P}{\mathbb{P}}
\newcommand{\Q}{\mathbb{Q}}
\newcommand{\R}{\mathbb{R}}

\newcommand{\Z}{\mathbb{Z}}

\def\J{\mathfrak{J}}
\def\dow{\begin{proof}}
\def\kdow{\end{proof}}
\def\kwadrat{\hfill$\square$}

\title{Secant cumulants and toric geometry}

\author{Mateusz Micha\l ek}
\address{
Max Planck Institute for Mathematics,
Vivatsgasse 7,
53111 Bonn,
Germany
\linebreak
Mathematical Institute of the Polish Academy of Sciences, \'{S}niadeckich 8, 00-956 Warszawa, Poland}
\email{wajcha2@poczta.onet.pl}
\thanks{The first author is supported by the Narodowe Centrum Nauki grant UMO-2011/01/N/ST1/05424.}
\author{Luke Oeding}
\address{Department of Mathematics, University of California, Berkeley, CA 94720-3840}
\email{oeding@math.berkeley.edu}
\thanks{The second author is partially supported by NSF RTG Award \# DMS-0943745}

\author{Piotr Zwiernik}
\address{Department of Statistics, University of California, Berkeley, CA 94720-3840}
\curraddr{}
\email{pzwiernik@berkeley.edu}
\thanks{The third author gratefully acknowledges the support from Jan Draisma's Vidi grant of the Netherlands Organisation for Scientific Research (NWO) and from the European
Union Seventh Framework Programme (FP7/2007-2013) under grant agreement PIOF-GA-2011-300975}
\date{\today}

\begin{document}
\maketitle
\begin{abstract}
We study the secant line variety of the Segre product of projective spaces using special cumulant coordinates adapted for secant varieties. We show that the secant variety is covered by open normal toric varieties. We prove that in cumulant coordinates its ideal is generated by binomial quadrics. We present new results on the local structure of the secant variety. In particular, we show that it has rational singularities and we give a description of the singular locus. We also classify all secant varieties that are Gorenstein.
Moreover, generalizing \cite{SturmfelsZwiernik}, we obtain analogous results for the tangential variety.
\end{abstract}
%\tableofcontents
\section{Introduction}

Cumulants are basic objects in probability and statistics used to describe probability distributions.  In this paper we use a variation of cumulants that we call \emph{secant cumulants} to study secant varieties of Segre products.
The Cartesian product of $n$ projective spaces $\P^{k_1}\times \dots \times \P^{k_n}$ embeds naturally in $\P^{N-1}$ ($N=\prod(k_i+1)$) via the Segre embedding. The secant variety of the Segre product (\emph{the secant variety} hereafter) is the Zariski closure of all points on all secant lines to the Segre product, and is denoted ${\rm Sec}\left(\P^{k_1}\times \dots \times \P^{k_n}\right)$. The tangential variety is, similarly, the union of all points on all tangent lines to the given variety.

The secant of the Segre embedding can also be described as the Zariski closure of the locus of tensors of rank two, or the tensors of \emph{border rank} $\leq 2$.
While the Segre variety and its higher secant varieties have been studied classically, there is also current interest in these topics due to the wide variety of applications.
For example, determining the rank and border rank of a tensor is connected to computational problems, such as fast matrix multiplication. The border rank of a tensor tells the minimal secant variety on which the tensor lives, and could be determined by explicit knowledge of the implicit defining equations of secant varieties. Finding these equations turns out to be a difficult problem in general.
The reader may consult \cite[Ch.~5]{Landsberg} for a modern account of the topic.

Garcia, Stillman and Sturmfels studied the secant of the Segre in the case of binary tensors  ($k_1=\ldots=k_n=1$) from the point of view of Bayesian networks, conjecturing that its defining ideal was generated by the $3\times 3$ minors of flattenings \cite[Conjecture~21]{GSS}.
Since then the so-called GSS conjecture had many partial solutions, including work of Allman and Rhodes \cite{AR}, Landsberg and Manivel \cite{LandsbergManivel}, Landsberg and Weyman \cite{LandsbergWeyman}.
 The GSS conjecture was finally resolved by Raicu \cite{RaicuGSS}, who also proved analogous results in the partially symmetric (Segre-Veronese) case.
%See \cite{RaicuGSS} for a complete account of the proof of the GSS conjecture and its generalization to the Segre-Veronese case, as well as the history of the problem and previous results.

Like secants, tangential varieties were also studied by classical geometers (see, for instance \cite{Zak}).
Its stratification by tensor rank, varying from $1$ to $n$, was given in \cite{Bernardi} and is also contained in the more general work \cite{BuczynskiLandsberg}.
They are also connected to Algebraic Statistics because of their interpretation as special context-specific independence models (see \cite[\S~6]{Oeding_tan}).
%The tangential variety is better understood than the secant. In particular, i
Speaking of algebraic properties,  it was known that it is always arithmetically Cohen-Macaulay \cite[Theorem~7.3]{LandsbergWeyman2}.
 Landsberg and Weyman studied the equations of tangential varieties of compact Hermitian symmetric spaces, and, in particular they conjectured the defining equations of the tangential variety of the Segre (see \cite[Conjecture~7.6]{LandsbergWeyman2}). Their conjecture was proved set theoretically by one of us in \cite{Oeding_tan}, and ideal theoretically in \cite{OedingRaicu}, both making use of representation theoretic methods.
Still we are not aware of any explicit description of the singular locus, apart from the fact that the variety is normal, hence the codimension of the singular locus is at least $2$.

Common methods for studying the secants and tangents of the Segre use the fact that there is a natural action of the product of general linear groups on the ambient space, which enables the application of methods of the representation theory of $GL(n)$. Our approach is completely different. It relates to toric geometry, however not, as one could expect, to the action of the dense torus orbit of the Segre variety. Our inspirations come from statistics (see \cite{SturmfelsZwiernik,pwz-2010-cumulants}), and from other uses of toric techniques in Algebraic Statistics and Phylogenetics such as \cite{SturmfelsSullivant05, SturmfelsSullivant06, SturmfelsSullivant08, BuczynskaWisniewski, DraismaKuttler, MateuszGBM}.

 Our idea is to study an affine open subset of the projective space, for which we can then use probabilistic and combinatorial techniques. Secant cumulants are well-defined on this open set of tensors and give a non-linear birational change of coordinates on projective space.
This change of coordinates enables us to explicitly provide a covering of the secant variety with affine toric varieties (Thm.~\ref{thm:bundle}), which are defined by quadratic binomials (Cor.~\ref{cor:quad}). Each of these varieties is a cone over a projective toric variety. These toric varieties are described by normal polytopes with unimodular regular triangulations. Adapting the arguments in the previous work of Sturmfels we can show that their ideals have a quadratic square-free Gr\"obner bases (Thm.~\ref{thm:deg2GB}). In particular, the varieties have rational singularities, thus are normal and Cohen-Macaulay (Thm.~\ref{cor:CM2}).
As a consequence, we also obtain a short proof of the (slightly weaker) scheme-theoretic version of Raicu's theorem (Thm.~\ref{thm:flat}). %\luke{Why don't we just say the following:}\luke{Since all of our methods and results are independent of the characteristic of the ground field, we choose to work over an algebraically closed field $\mathbb{K}$.} \luke{and change every instance of $\C$ to $\mathbb{K}$.}

In the remainder of this section we highlight (and slightly rephrase) our results on the secant.
Our methods also adapt to the case of the tangential variety, for which we find analogous results with straightforward proofs. We report on these in Section~\ref{sec:tangent}.

We consider the following our main result.
\newtheorem*{cor:CM2}{\bf Theorem~\ref{cor:CM2}}
\begin{cor:CM2}
The secant variety of the Segre product of projective spaces ${\rm Sec}\left(\P^{k_1}\times \dots \times \P^{k_n}\right)$ is covered by normal affine toric varieties. In particular it has rational singularities.
\end{cor:CM2}

%It is not immediately clear that an open dense subset of the secant variety should be toric, and might perhaps be unexpected, as the algebraic torus acting on the Segre is too small to give a dense orbit on the secant. However, secant cumulants now make this evident.

Toric geometry facilitates the following description of the singular locus. \newtheorem*{cor:components}{\bf Corollary~\ref{cor:components}}
\begin{cor:components}
The singular locus of the secant variety ${\rm Sec}\left(\P^{k_1}\times \dots \times \P^{k_n}\right)$ is \[\bigcup_{1\leq i_{1}<i_{2}\leq n}
\P^{k_{1}}\times\dots\times \widehat{\P^{k_{i_{1}}}} \times \dots \times \widehat{\P^{k_{i_{2}}}} \times\dots \times \P^{k_{n}} \times {\rm Sec}(\P^{k_{i_{1}}}\times\P^{k_{i_{2}}}),\]
where $\widehat{\cdot}$ denotes omission.
\end{cor:components}

Moreover, we classify all the the secant varieties that are Gorenstein as follows:

 \newtheorem*{thm:secGor}{\bf Theorem~\ref{thm:secGor}}
\begin{thm:secGor}
Assume that $k_1\leq k_2\leq\dots \leq k_n$.
The secant variety ${\rm Sec}\left(\P^{k_1}\times \dots \times \P^{k_n}\right)$ is Gorenstein only in the following cases:
\begin{enumerate}
\item $n=5$ and $k_1=k_2=k_3=k_4=k_5=1$,
\item $n=3$ and $(k_{1},k_{2},k_{3})$ equal to one of $(1,1,1)$, $(1,1,3)$, $(1,3,3)$, or $(3,3,3)$,
\item $n=2$ and $k_1=k_2$ or $k_1=1$, $k_{2}$ arbitrary.
\end{enumerate}
\end{thm:secGor}
%\piotr{Follow Nathan's comments}
%We have also studied the tangential variety in parallel to our study of secants, and obtain similar results to each of the preceding results.  See Section~\ref{sec:tangent} for more details.
%
Here is an outline for the rest of the paper.
In Section~\ref{sec:cumulants} we describe secant cumulants. To keep the notation as simple as possible we first present the case of the Segre product of projective lines. In Section~\ref{sec:secant} we study the secant variety in secant cumulants. In Section~\ref{sec:toric} we describe the toric varieties that give an open covering of the secant variety and further study the geometry of the secant variety in Section~\ref{sec:pultriangsingloc}. In Section~\ref{sec:flats} we show the explicit connection between equations in secant cumulants to minors of flattenings. Then we show how this setting can be generalized. In Section~\ref{sec:generalizations} we adopt our results to the secant variety of the Segre embedding of arbitrary projective spaces. In Section~\ref{sec:tangent} we present analogous results for the tangential variety.
%\luke{perhaps we omit this as mentioned above:}
Most of our results hold over a field of arbitrary characteristic. However, some of them concern notions typical for characteristic zero, like rational singularities. Thus, for simplicity, we work over the complex ground field $\C$.

\section{Secant cumulants}\label{sec:cumulants}
In this section we introduce \emph{secant cumulants}
as a special kind of $L$-cumulants described in \cite{pwz-2010-cumulants}. This is a purpose-built coordinate system suitable for studying the secant ${\rm Sec}((\P^1)^{\times n})\subseteq \P^{2^n-1}$. After setting up notation we will make this change of coordinates on $\P^{2^n-1}$ in two steps: from moments $x_{I}$, to central moments $y_{I}$ then to secant cumulants $z_{I}$, where all indices $I$ are subsets of $[n]$. The interested reader can consult \cite{SturmfelsZwiernik,pwz-2010-cumulants} for more statistical background.

We say that $\pi=B_{1}|\ldots|B_{k}$ is a \emph{set partition} (or \emph{partition}) of $[n]:=\{1,\ldots,n\}$, if the \emph{blocks} $B_{i}\neq \emptyset$ are  disjoint sets whose union is $[n]$. In a similar way we define any set partition of any finite set. We are interested in very special type of set partitions.
\begin{defn}\label{def:set partitions} An \emph{interval set partition} of $[n]$ is a set partition $\pi$ of a form
\[1\cdots i_{1}|(i_{1}+1)\cdots i_{2}|\cdots|(i_{k}+1)\cdots n,\]
for some $0\leq k\leq n-1$ and $1\leq i_{1}<\ldots<i_{k}\leq n-1$. Denote the  poset of all interval set partitions by ${\rm IP}({[n]})$. For any $I\subset[n]$ denote by ${\rm IP}(I)$ the poset of interval partitions of $I$ induced from ${\rm IP}({[n]})$ by constraining each partition to elements of $I$. There is an order-preserving bijection between the poset  ${\rm IP}({[n]})$ and the Boolean lattice of subsets of $[n-1]$.
\end{defn}
For example ${\rm IP}(\{1,2,3\})$ consists of four set partitions: $123$, $1|23$, $12|3$ and $1|2|3$.
\begin{rem}\label{rem:order}Note that in Definition~\ref{def:set partitions} we used the total ordering $1<2<\cdots<n$. Other total orderings lead to other interval set partition lattices. For example for  ordering $3<1<2$ the corresponding interval set partitions are: $123$, $3|12$, $13|2$ and $1|2|3$. This will play a role in Section~\ref{sec:flats}.
\end{rem}

Let $x_I$ for $I\subseteq [n]$ be the coordinates of $\P^{2^n-1}$. Write $x_i$ for $x_{\{i\}}$, $x_{ij}$ for $x_{\{i,j\}}$ and so on.  We now construct the secant cumulants $z_I$. Let $U_{\emptyset}$ denote the affine subset given by $x_\emptyset=1$. We first write all the coordinate changes in $U_{\emptyset}$. The first change of coordinates $U_{\emptyset}\rightarrow U_{\emptyset}$ is defined by
\begin{equation}\label{eq:cremonaxtoy}
\begin{array}{lcll}
y_{i}&:=&x_{i}, &\mbox{for }i=1,\ldots,n,\\
y_I&:=&\sum_{A\subseteq I}(-1)^{|I\setminus A|}x_A\prod_{i\in I\setminus A}x_i,&\mbox{for all }I\subseteq [n],\text{ s.t. }|I|\geq 2.
\end{array}\end{equation}
By construction the change of coordinates is triangular, that is $y_I=x_I+{\rm lower \, terms}$ for every $I\subseteq [n]$. Hence it forms an isomorphism  between $\C\left[x_I:\,I\subseteq [n]\right]$ and $\C\left[y_I:\,  I\subseteq [n]\right]$ where we set $y_\emptyset=x_\emptyset=1$.

Now we define $U_{\emptyset}\rightarrow U_{\emptyset}$ by
\begin{equation}\label{eq:cremonaytoz}
\begin{array}{lcll}
 z_{i}&:=&y_{i},& \mbox{for }i=1,\ldots,n,\\
z_I&:=&\sum_{\pi\in {\rm IP}(I)}(-1)^{|\pi|-1}\prod_{B\in \pi}y_B &\mbox{for }I\subseteq [n], \text{ s.t. } |I|\geq 2,
\end{array}\end{equation}
where the sum runs over all interval set partitions without singleton blocks. Again the change of coordinates is triangular, hence an isomorphism
 between $\C[y_I:\, I\subseteq [n]]$ and $\C[z_I:\, I\subseteq [n]]$ with $z_\emptyset=y_\emptyset=1$.
 In both cases explicit forms of the inverse maps can be given by the M\"{o}bius inversion formula.

\begin{exmp}
If $n=3$ then
$y_i=x_i$, $y_{ij}=x_{ij}-x_ix_j$ for $i\neq j\in \{1,2,3\}$ and
\[y_{123}\quad =\quad x_{123}-x_1x_{23}-x_2x_{13}-x_3x_{12}+2x_1x_2x_3.\]
Moreover, $z_i=y_i$, $z_{ij}=y_{ij}$, $z_{123}=y_{123}$ so the second change of coordinates is just the identity. The inverse maps are given by $x_{ij}=y_{ij}+y_iy_j$ and
\[x_{123}\quad=\quad y_{123}+y_1y_{23}+y_2y_{13}+y_3y_{12}+y_1y_2y_3.\]

If $n=4$ then $y_{ij}=x_{ij}-x_ix_j$, $y_{ijk}=x_{ijk}-x_ix_{jk}-x_jx_{ik}-x_kx_{ij}+2x_ix_jx_k$ for distinct $i,j,k\in \{1,2,3,4\}$; and
\begin{eqnarray*}
y_{1234}&=&x_{1234}-x_{1}x_{234}-x_2x_{134}-x_3x_{124}-x_4x_{123}+\\&+& x_{12}x_3x_4+x_{13}x_2x_4+x_{14}x_2x_3+x_{23}x_1x_4+x_{24}x_1x_3+x_{34}x_1x_2-3x_1x_2x_3x_4.
\end{eqnarray*}
Moreover, $z_{ij}=y_{ij}$, $z_{ijk}=y_{ijk}$ and $z_{1234}=y_{1234}-y_{12}y_{34}$.
\end{exmp}

\section{The secant variety of the Segre variety}\label{sec:secant}

 In this section we define the secant variety. When expressed in secant cumulants it will reveal its nice local structure. Consider the product of projective lines in the Segre embedding given by

\begin{eqnarray*}
\P^{1}\times \dots \times \P^{1} &\to& \P^{2^{n}-1}\\
\left([a_{1}^{1},a_{1}^{2}],\dots,[a_{n}^{1},a_{n}^{2}] \right)&\mapsto&
\left[
x_I\;=\; \prod_{i\in I}a_i^{1} \prod_{i\not\in I}a_{i}^{2}
\right]
.\end{eqnarray*}
On the affine open set $U_{\emptyset}$ we can assume that $[a_{i}^{1},a_{i}^{2}] = [a_{i},1]$ and the Segre embedding is parameterized by
\[
x_I\quad=\quad \prod_{i\in I}a_i, \quad\quad\text{for all}\quad I\subseteq [n].
\]

The secant variety to a variety $X$, denoted ${\rm Sec}(X)$, is the Zariski closure of all lines connecting pairs of points on the variety.
On the open set $U_{\emptyset}$, the secant variety ${\rm Sec}((\P^{1})^{\times n})$ is parameterized by
\[
x_I\quad=\quad (1-t)\prod_{i\in I}a_i+t\prod_{i\in I}b_i,\qquad\mbox{for all } I\subseteq [n]
,\]
where $a_{i}$ and $b_{i}$ are $\C$ valued parameters. We introduce the affine variety $V$ given by
\[
V\quad:=\quad {\rm Sec}((\P^{1})^{\times n}) \cap U_{\emptyset}.
\]
By $\cI(V)$ denote its defining ideal. We will see that in secant cumulants, $V$ has a monomial parameterization. To show this we first prove the following result.
\begin{lem}\label{lem:parametryzacja}
The variety $V$ in the coordinate system given by the secant cumulants is the Zariski closure of the image of the parameterization given by:
\begin{eqnarray*}
z_i&=&(1-t)a_i+tb_i,\,\quad\qquad\qquad\qquad\qquad\mbox{for all }\,i\,=\,1,\ldots,n, \,\mbox{ and }\\
z_I&=& t(1-t)(1-2t)^{|I|-2}\prod_{i\in I}(b_i-a_i)\qquad\text{ for }\,|I|\,\geq\, 2.
\end{eqnarray*}
\end{lem}
\dow
Consider a point of the secant given by
\[
x_I\quad=\quad (1-t)\prod_{i\in I}a_i+t\prod_{i\in I}b_i,\qquad\mbox{for all } I\subseteq [n].
\]
The formula for $z_i$ follows directly from the fact that $z_i=x_i$ for all $i=1,\ldots,n$. We therefore focus on the case $|I|\geq 2$.
First we will prove that $y_I$ vanishes for $|I|\geq 2$ if $a_i=b_i$ for some $i\in I$. Then we will show that the remaining factors in the expression of $y_{I}$ only depend on $t$ and give the precise expression. Finally, and in a similar fashion, we convert the expression for $y_{I}$ to the resulting expression for $z_{I}$.

Consider $i$ fixed, $|I|\geq 2$ and a subset of $A\subseteq I$, such that $i\in A$ and $a_i=b_i$. The corresponding term in the expression of $y_{I}$ satisfies
\begin{eqnarray*}
x_A\prod_{j\in I\setminus A}x_j
&=&((1-t)\prod_{j\in A}a_j+t\prod_{j\in A}b_j)\prod_{j\in I\setminus A}x_j
\\
&=&a_i((1-t)\prod_{j\in A\setminus i}a_j+t\prod_{j\in A\setminus i}b_j)\prod_{j\in I\setminus A}x_j
\\
&=&
((1-t)\prod_{j\in A\setminus i}a_j+t\prod_{j\in A\setminus i}b_j)((1-t)a_i+tb_i)\prod_{j\in I\setminus A}x_i
\\
&=&x_{A\setminus i}\prod_{j\in I\setminus (A\setminus i)}x_i.
\end{eqnarray*}
We can pair the subsets indexing term in the sum in the expression for $y_{I}$ in \eqref{eq:cremonaxtoy}
 by $(B,B\setminus i)$, where $B$ contains $i$. From the previous computation we see that the sum in each pair will be zero, hence $y_I=0$. Thus $y_I=f_I(a_j,b_j,t)\prod_{i\in I}(b_i-a_i)$, for some polynomial $f_I$. Notice that $y_I$ is of  degree $|I|$ in variables $a_j,b_j$ for $j\in I$. Hence $f_I$ depends only on the variable $t$. To determine $f_I$ set all $a_i=0$ and $b_i=1$. Then
\begin{eqnarray*}
f_I(t)
&=&
\sum_{\emptyset\neq A\subseteq I}(-1)^{|I\setminus A|}t\prod_{i\in I\setminus A}t+(-1)^{|I|}\prod_{i\in I}t
\\
&=&
\sum_{\emptyset\neq A\subseteq I}(-1)^{|I\setminus A|}t^{|I\setminus A|+1}+(-1)^{|I|}t^{|I|}
\\
&=&
(-t)^{|I|}+\sum_{k=1}^{|I|}(-1)^{|I|-k}{{|I|}\choose k} t^{|I|-k+1}
\\
&=&(-t)^{|I|}-(-1)^{|I|}t^{|I|+1}+\sum_{k=0}^{|I|}(-1)^{|I|-k}{{|I|}\choose k} t^{|I|-k+1}
\\
&=&(-t)^{|I|}(1-t)+t(1-t)^{|I|}.
\end{eqnarray*}

As $f_I$ depends only on the size of $I$ we will denote it by $f_{|I|}$.
Substituting $y_I$ in the definition of $z_I$ we see that $z_I=h_{|I|}(t)\prod_{i\in I}(b_i-a_i)$ for some polynomial $h_{|I|}(t)$. Let us prove inductively on the size of $I$ that $h_{|I|}(t)=t(1-t)(1-2t)^{|I|-2}$. The case $|I|=2$ can be easily checked by hand. Let $m:=|I|$.
By induction, assume that the result holds for sets of cardinality strictly smaller than $m$.  We have
\[h_{|I|}(t)=(-t)^m(1-t)+t(1-t)^m-\sum_{i=2}^{m-2}\sum_{\pi:I-i}f_i(t)h_{m-i}(t).\]
Here the two first terms correspond to the partition of $I$ into one set. The sum runs over nontrivial interval partitions, where $i$ denotes the size of the first set in the interval partition and the second sum runs over interval partitions of $I$ without the first $i$ elements. By the inductive assumption we have:
\begin{eqnarray*}
-\sum_{i=2}^{m-2}\sum_{\pi:I-i}f_i(t)h_{m-i}(t)
&=&
t^2(1-t)^2\sum_{i=2}^{m-2}((-t)^{i-1}-(1-t)^{i-1})(1-2t)^{m-i-2}
\\
&=&
t^2(t-1)^2\sum_{i=0}^{m-3}((-t)^{i}-(1-t)^{i})(1-2t)^{m-i-3}
\end{eqnarray*}
Notice that
\begin{eqnarray*}(1-t)(\sum_{i=0}^{m-3}(-t)^i(1-2t)^{m-i-3})
&=&((1-2t)-(-t))(\sum_{i=0}^{m-3}(-t)^i(1-2t)^{m-i-3})\\
&=&(1-2t)^{m-2}-(-t)^{m-2}
\end{eqnarray*}
and
\begin{eqnarray*}(-t)(\sum_{i=0}^{m-3}(1-2t)^{m-i-3}(1-t)^i)
&=&(1-2t-(1-t))(\sum_{i=0}^{m-3}(1-2t)^{m-i-3}(1-t)^i)\\
&=&
(1-2t)^{m-2}-(1-t)^{m-2}.
\end{eqnarray*}
Substituting this, we easily prove the inductive step.
\kdow

%%%%%%%%%%%%%%%%%%%%%%%%%%%%%
\section{The toric varieties $T_{a,b}$ and the ideal $\cI(V)$}\label{sec:toric}
We define a special toric variety, which is closely related to the secant variety.

\begin{defn}[$T_{a,b}$, $\J_{a,b}$]\label{J}
Fix three integers $0\leq a\leq b\leq n$. Consider the lattice $\Z^{n+1}$ and the set $\J_{a,b}$ consisting of all points $p$ with the following properties:
\begin{enumerate}
\item  $p\in \{1\}\times\{0,1\}^{n}$,
\item  $a+1\leq \#p \leq b+1$, where $\#p$ denotes the number of non-zero coordinates of $p$.\end{enumerate}
We define the affine toric variety $T_{a,b}^n$ to be the spectrum of the semigroup algebra associated to the monoid generated by $\J_{a,b}$. The reader may wish to consult \cite[Ch.~13]{Sturmfels96} for more details on this type of construction. By $P_{a,b}^n$ denote the associated polytope in $\R^{n+1}$ given as the convex hull of points in $\J_{a,b}$. Typically $n$ is fixed and we omit the superscript so that $T_{a,b}:=T_{a,b}^n$ and $P_{a,b}:=P_{a,b}^n$.
\end{defn}
Various versions of the varieties $T_{a,b}$ have already appeared in the literature (with or without the homogenizing condition).
For $a=0$ and $b=n$ we obtain the affine cone over the Segre variety. For $a=b=2$ we obtain toric varieties arising from complete graphs \cite{Hibi1}. For $a=2$ and $b=n$ without the homogenizing condition we obtain a variety related to the tangential variety of the Segre \cite[Theorem~4.1]{SturmfelsZwiernik} also studied in \cite{GrossPetrovic}.
\begin{rem}\label{rem:exSturm}
The polytope $P_{a,b}$ is a special case of \cite[Section 14A]{Sturmfels96}. Indeed, using the notation from the book, up to lattice isomorphism it corresponds to the set $\mathcal{A}$ for $d:=n+1$, $s_1,\dots,s_n:=1$, $s_{n+1}:=b-a$, $r:=b$.
\end{rem}
Each point of $\J_{a,b}$ can be represented by a subset of $[n]$ of indices on which it is nonzero. Thus points of $\J_{a,b}$ correspond to subsets of cardinality at least $a$ and at most $b$.
Notice that the variety $T_{a,b}$  is the Zariski closure of the image of the map
\begin{eqnarray*}
(\C^*)^{n+1}&\to& \C^{|\J_{a,b}|}\\
(t_0,t_1,\ldots,t_n) &\mapsto& \left[z_I\,\,=\,\, t_0 \prod_{i\in I} t_i\right],
\qquad\mbox{for all } a\leq |I|\leq b.
\end{eqnarray*}
The variety $T_{a,b}$ is a cone over a projective variety due to condition (1). The projectivization of $T_{a,b}$ will be denoted by $\P(T_{a,b})$.
Notice that
\[\dim(T_{a,b}^{n}) = n+1 \qquad\mbox{ for } a\neq b\]
 because the polytope is full dimensional, hence so is the cone over it, and so is the toric variety. We also have $\dim T_{a,a}^n=n$ for $a\neq n$.

\begin{exmp}\label{ex:T23}
For $T_{2,3}$ the subset $\J_{2,3}$ consists of points
\[(1,1,1,0);(1,1,0,1);(1,0,1,1);(1,1,1,1),\]
and thus its convex hull is a simplex. The corresponding subsets are respectively:
\[\{1,2\},\{1,3\},\{2,3\},\{1,2,3\}.\]
\end{exmp}
\begin{exmp}
For $T_{1,2}$ the subset $\J_{1,2}$ consists of points
\[(1,1,0,0);(1,0,1,0);(1,0,0,1),
\qquad(1,1,1,0);(1,1,0,1);(1,0,1,1),\]
and thus its convex hull is an octahedron. The corresponding subsets are respectively:
\[\{1\}, \{2\}, \{3\}
\qquad
\{1,2\},\{1,3\},\{2,3\}.\]
\end{exmp}

The following theorem is the main result of this section.

\begin{thm}\label{thm:bundle}
The variety $V$ is the trivial affine bundle of rank $n$ over the variety $T_{2,n}$. In particular the secant variety is locally isomorphic to the trivial affine bundle of rank $n$ over the variety $T_{2,n}$.
\end{thm}
\dow
Introduce variables $d_i=(b_i-a_i)(1-2t)$ and $t'=t(1-t)/(1-2t)^2$. The parameterization from Lemma~\ref{lem:parametryzacja} is given by $z_I=t'\prod_{i\in I} d_i$ for $|I|\geq 2$. Moreover, $z_I$ depends on $a_i, b_i$ only through $b_i-a_i$. Hence we see that $z_i$ can be arbitrary, by varying $a_i$ and keeping the difference fixed. For the last statement, note that the secant variety can be covered by varieties isomorphic to a trivial vector bundle over $T_{2,n}$ by taking different hyperplanes $x_I\neq 0$.
\kdow

This result motivates a further study of the toric variety $T_{a,b}$ and hence also the ideal $\cI(V)$.
We now show that the ideal of $T_{a,b}$ is generated by very special quadrics. For a point $p$ we denote by $p_i$ its $i$-th coordinate.
\begin{defn}[bumping, swapping]\label{def:bumswap}
Let $a\leq b$ be fixed integers. Denote by $e_i\in \{0,1\}^{n+1}$ the unit vector with zeros everywhere apart from position corresponding to $i\in \{0\}\cup [n]$.
\begin{enumerate}
\item Suppose that $p,q\in \J_{a,b-1}$. If  $p_i=q_i=0$ for some $i\in [n]$ then all four points $p,q,p+e_i,q+e_i$ lie in $\J_{a,b}$. The obvious relation holds:
\[
(p+e_i)+q=p+(q+e_i).
\]
We call this relation \emph{bumping}.
\item Suppose that $p,q\in \J_{a-1,b-1}$ and there exist two elements $i,j\in [n]$ such that $p_i=p_j=q_i=q_j=0$. Then all four points $p+e_i,p+e_j,q+e_i,q+e_j$ lie in $\J_{a,b}$ and
\[
(p+e_i)+(q+e_j)\quad=\quad(p+e_j)+(q+e_i).
\]
We call this relation \emph{swapping}.
\end{enumerate}\end{defn}

Every relation \[\sum_{i=1}^d p^i \quad=\quad \sum_{i=1}^d q^i\] among points $p^{i},q^{i}$ in $\J_{a,b}$ induces a binomial relation
\[
\prod_{i=1}^d z_{I_i} \quad=\quad \prod_{i=1}^d z_{J_i},
\]
where $I_i, J_i$ are subsets of $[n]$ corresponding to points $p^i$ and $q^i$ respectively. It is a well-known fact that the polynomials in the ideal of $T_{a,b}$ are linear combinations of binomials of the above form  -- see \cite[Lemma 4.1]{Sturmfels96}. Two important examples of such relations are the bumping and swapping relation of Definition ~\ref{def:bumswap}. The bumping relation corresponds to a binomial of the form
\[
z_{I\cup \{i\}}z_J\,\, =\,\, z_I z_{J\cup \{i\}},\qquad\mbox{where } a\leq |I|,|J|\leq b-1,\, i\notin I\cup J,
\]
and the swapping relation to
\[
z_{I\cup \{i\}}z_{J\cup \{j\}}\,\, =\,\, z_{I\cup \{j\}} z_{J\cup \{i\}},\qquad\mbox{where }a-1\leq |I|,|J|\leq b-1,\, i,j\notin I\cup J.
\]
Their importance is due to the following proposition.

\begin{prop}\label{prop:Tabideal}
The ideal of the variety $T_{a,b}$ is generated by quadrics corresponding to \emph{bumping} and \emph{swapping}. In particular the ideal $\mathcal{I}(V)$ of $V$ is the ideal generated by bumping and swapping relations in $\J_{2,n}$.
\end{prop}
\dow
Consider any relation
\[\sum_{i=1}^d p^i\,\,=\,\,\sum_{i=1}^d q^i \]
between points $p^{i}, q^{i} \in \J_{a,b}$.
 We have the same number of summands on both sides because $p_0=1$ for every point $p$ in $\J_{a,b}$. By bumping we can assume that all points $p^i$ and $q^i$ have exactly $m$ or $m+1$ nonzero entries for some $a+1\leq m\leq b+1$. Moreover we can assume that $p^1$ and $q^1$ have exactly $m$ nonzero entries.
Write the vector $\sum p^i$ as $(d,a_1,\ldots,a_n)\in \Z^{n+1}$. Without loss we can assume $a_1\geq a_2\geq \cdots\geq a_n$. Now we show that $p^1$ can be transformed to $(1,1,\ldots,1,0,\ldots,0)$ with $m$ ones. Suppose $p^1_i=0$ for $i\leq m$. Then $p^1_j=1$ for some $j>m$. It follows by $a_i\geq a_j$ that there exists $k$ such that $p^k_i=1$, $p^k_j=0$. Swap. Pick new $i$, $j$ and swap recursively until done. The same argument applies to $q_1$, which allows us to decrease the degree of the relation and finishes the proof.
\kdow

In the special case when $a=2$ and $b=n$ we obtain the following simpler set of generators.
\begin{cor}\label{cor:quad}
The ideal of $T_{2,n}$, and hence also $\cI(V)$, is generated by all bumping relations:
$z_{I}z_{J\cup\{j\}}=z_{J}z_{I\cup\{j\}}$ for $j\not\in I\cup J$ and $|I|,|J|>1$, together with a subset of swapping relations of the form: $z_{ij}z_{kl}=z_{il}z_{jk}$ for $i,j,k,l$ all distinct.
\end{cor}
\dow
By Proposition~\ref{prop:Tabideal}  it is enough to show that the remaining swapping relations can be generated from the provided binomials. Notice that when one of two sets has cardinality at least $3$ then swapping can be generated by two consecutive bumpings. When both sets are of cardinality $2$ we obtain the second relation.\kdow

The next Proposition follows from \cite[Theorem 14.2]{Sturmfels96} and Remark~\ref{rem:exSturm}.
\begin{prop}\label{gentor}
There exists a term order for which $\cI(V)$ has a quadratic square-free Gr\"obner basis.\kwadrat
\end{prop}

This Proposition gives us the following general characterization of the singular locus of the secant. In Corollary~\ref{cor:singp1} we will provide a more in-depth analysis.
\begin{cor}\label{cor:CM}
The secant variety ${\rm Sec}((\P^1)^{\times n})$ has rational singularities. In particular it is normal, Cohen-Macaulay and has singular locus of codimension at least $2$.
\end{cor}\dow
 By Proposition~\ref{gentor} there exists a square-free Gr\"obner basis. By \cite[Corollary 8.9]{Sturmfels96} this induces a unimodular triangulation of the polytope $P_{2,n}$. In particular, $P_{2,n}$ is normal, thus so is $T_{2,n}$ -- \cite[Proposition 13.15]{Sturmfels96}. By \cite[Theorem~11.4.2]{Cox} it has rational singularities, in particular Cohen-Macaulay. Because all of the notions are local the corollary follows.
\kdow
\begin{rem}In
Theorem~\ref{cor:CM2} we give an easy adaptation of Corollary~\ref{cor:CM} to the secant variety of Segre products of projective spaces of arbitrary size.
\end{rem}

In the next section we use tools of toric geometry to analyze local properties of the secant in more detail.

\section{The singular locus of the secant}\label{sec:pultriangsingloc}
Let us further study the toric geometry of the obtained variety. As we know that the polytope is normal, it is natural to describe its fan.
For what follows recall $\J_{a,b}$ from Def.~\ref{J}.
\begin{lem}\label{prop:fan}
Consider the projection of $P_{2,n}\subseteq \R^{n+1}$ to $\R^n$ by forgetting the first coordinate.
Denote by $Q$ the image of $P_{2,n}$. Then $Q$ is given by the points $(q_{i})\in \R^{n}$ satisfying the following inequalities:
\begin{itemize}
\item $0\leq q_i\leq 1$,
\item $\sum_{i=1}^n q_i\geq 2$.
\end{itemize}
Moreover, if $n\geq 4$, these inequalities provide the minimal facet description of the polytope.
\end{lem}
\dow
The inequality description is a special case of Proposition~\ref{prop:facetdesc}. The remaining thing is to show that for $n\geq 4$ each of the defining inequalities of $Q$ corresponds to a facet. We prove it easily by checking that for each inequality, the affine combination of the set of points in $\mathfrak{J}_{2,n}$ satisfying this inequality as equality is a linear subspace of codimension $1$.
\kdow

Lemma~\ref{prop:fan} gives us immediately the description of $P_{2,n}$. Let us describe the toric divisors associated to each facet of $P_{2,n}$. Each lattice point of the polytope corresponds to a coordinate of the embedded affine space. Fix a face $F$. Recall, that a toric variety associated to $F$ is an intersection of $T_{2,n}$ with the linear space
defined by vanishing of all the variables not belonging to $F$. The polytope associated to the intersection is exactly the face $F$. If $F$ is a facet, we obtain in this way a divisor and all toric divisors are of this form.
\begin{prop}\label{prop:toricdivisors}
The toric divisors of $T^{n}_{2,n}$  are each isomorphic to one of the following:
\begin{enumerate}
\item
$T_{2,n-1}^{n-1}$, associated to the facet $q_i=0$,
\item
$T_{2,2}^n$, associated to the facet $\sum q_i= 2$,
\item
$T_{1,n-1}^{n-1}$, associated to the facet $q_i=1$.
\end{enumerate}
We give the explicit parameterizations in the proof.
\end{prop}
\dow
The isomorphisms of the divisors with given varieties follow directly from the description of the facets.
Let us give the descriptions of the paremeterizations.
\begin{enumerate}
\item The facet $q_i=0$:

The divisor contains exactly those points of $T_{2,n}$ that are equal to zero on the coordinates parameterized by monomials containing $d_i=(b_i-a_i)(1-2t)$ with nonzero exponent. Thus it can be parameterized by setting $d_i=0$, and therefore
the parameterization of this variety in original coordinates is obtained by restricting the original parameterization of the secant:
\[(t,a_1,\dots,a_n,b_1,\dots,b_n)\rightarrow  (t\prod_{j\in I}a_j+(1-t)\prod_{j\in I}b_j)  = x_{I}\]
to the subspace $a_i=b_i$.

\item The facet $\sum q_i= 2$:

In this case we are setting to zero those $z_{I}$-coordinates that correspond to points with at least three nonzero entries. Recall that the parameterization of the secant in cumulant coordinates is given by
\[(t,a_1,\dots,a_n,b_1,\dots,b_n)\rightarrow t(1-t)(1-2t)^{|I|-2}\prod_{i\in I}(b_i-a_i) =z_{I}\]
    so the image for $t=1/2$ is indeed contained in the divisor. As it is irreducible and of the right dimension, its Zariski closure must coincide with the divisor.
In particular the intersection of the divisor with a given open affine set is given by midpoints of segments joining two points of the Segre variety.

\item The facet $q_{i}=1$:

Consider the full affine parameterization of the affine cone over the secant:
\[(t_1,t_2,a_i^1,a_i^2,b_i^1,b_i^2)\rightarrow \left(t_1\prod_{i\in I}a_i^1\prod_{i\notin I}a_i^2+t_2\prod_{i\in I}b_i^1\prod_{i\notin I}b_i^2 \right) = x_{I}.\]

 We will show that the divisor is the closure of the image of the restriction of the parameterization to $b_i^2=0$ (the closure of the image is the same if we restrict to $a_i^2=0$).

Consider any point of the given parameterization that is in $U_\emptyset$. We claim that $y_I=0$ for $i\not\in I$. Indeed for $i\not\in I$ we have $x_I=t_1\prod_{j\in I} a_j\prod_{k\not\in I} a_k$. So indeed $y_I=0$ -- one can check it either by direct computation or by the fact that on such $I$ the point coincides with a point of the Segre. Hence, a fortiori, $z_I=0$ for $i\not\in I$. But this is a condition of our divisor, so the image of the parameterization belongs to the divisor. By the dimension argument, the closure of the parameterization map must be equal to the divisor.\qedhere
\end{enumerate}
\kdow

The polytope $P_{2,n}$ induces the following polyhedral fan.
\begin{defn}
The fan $\Sigma_{2,n}$ of the toric variety $\P(T_{2,n})$ consists of $2^n-n-1$ maximal polyhedral cones and their subcones. The maximal cones are constructed as follows.
For each vertex $v\in \J_{2,n}$ consider normal vectors to all facets containing $v$ pointing inside the polytope. The corresponding polyhedral cone generated by these vectors is denoted by $\sigma_v$.
\end{defn}
Suppose $n\geq 4$. In this case $P_{2,n}$ has exactly $2n+1$ faces given by inequalities in Lemma~\ref{prop:fan}. Every vertex $v$ of $\J_{2,n}$ such that $|v|>2$ lies in exactly $n$ facets of $P_{2,n}$. In addition there are ${n\choose 2}$ vertices satisfying $|v|=2$, which lie in $n+1$ facets. If $|v|>2$ then
\[\sigma_v={\rm cone}((-1)^{v_1}e_1,\ldots,(-1)^{v_n}e_n),\]
where $(e_i)$ denotes the standard basis of $\R^n$, and hence $\sigma_v$ is one of the  orthants of $\R^n$ and thus smooth. If $|v|=2$ then
\[\sigma_v={\rm cone}(e_1+\cdots+e_n,(-1)^{v_1}e_1,\ldots,(-1)^{v_n}e_n).\]
In particular it is not simplicial as there are $n+1$ rays.
We have just proved the following lemma.
\begin{lem}\label{lem:fancones}
Let $n\geq 4$ and consider the polyhedral fan $\Sigma_{2,n}$. If $v\in \J_{2,n}$ is such that $|v|>2$ then the corresponding cone $\sigma_v$ of $\Sigma_{2,n}$ is smooth. If $|v|=2$ then $\sigma_v$ is not simplicial.
\end{lem}

This analysis provides a precise description of the singular locus of $T_{2,n}$.

\begin{prop}\label{prop:codimsing}
For $n=2$ and $n=3$ the variety $T_{2,n}$ fills the whole ambient space, hence it is smooth. For $n\geq 4$ the singular locus of $T_{2,n}$ has codimension equal to $n$. It consists of ${n}\choose{2}$ maximal dimensional components. In particular, singular locus is always of codimension at least $4$.\end{prop}
\dow
For $n=2$ the statement is trivial. For $n=3$ the polytope is the simplex from Example~\ref{ex:T23}. In particular, $T_{2,3}$ is the $4$ dimensional affine plane filling the whole space. Suppose now that $n\geq 4$ and consider the fan $\Sigma_{2,n}$. Note that we never have two vectors $e_i$ and $-e_i$ in one cone of the fan. First let us prove that cones of dimension smaller than $n$ are smooth. For sure they are smooth if they consist only of vectors of type $\pm e_i$. Suppose that the cone contains $e_1+\dots+e_n$ and $r$ vectors of type $\pm e_i$, where $r\leq n-1$. Since it never contains both $e_i$ and $-e_i$, each such cone is smooth. Thus we only have to study cones of dimension $n$. By Lemma~\ref{lem:fancones}, $\sigma_v$ is smooth whenever $|v|>2$ and there are ${n\choose 2}$ non-simplicial cones $\sigma_v$ corresponding to $|v|=2$. Thus, by \cite[Prop. 11.1.2]{Cox} the projective variety has exactly $n\choose 2$ singular points and the affine variety has $n\choose 2$ singular lines.
\kdow

Proposition~\ref{prop:codimsing} describes the singular locus of the variety $T_{2,n}$, not the secant. Hypothetically, it may happen that some components of the singular locus  of the secant are contained in the hyperplane section $x_\emptyset=0$. As we will easily prove, this is not the case.

\begin{cor}
For $n\geq 4$, the singular locus of the secant variety has codimension exactly $n$. It has $n\choose 2$ components, each two intersecting precisely on the Segre variety. Specifically, each component of the singular locus is isomorphic to $\P^3\times (\P^1)^{\times (n-2)}\simeq {\rm Sec}(\P^1\times\P^1)\times (\P^1)^{\times (n-2)}$.
\end{cor}
\dow
First we prove the statements regarding the number of components and their codimensions. Then we will give a parameterization of each of the components, proving the last statement.

Proposition~\ref{prop:codimsing} provides a lower bound for the number of components of the singular locus (those outside the hyperplane $x_{I}= 0$) and tells the dimensions of each of them.
We now show that there are no other components contained in the hyperplane $x_{I}=0$.

Notice that each component of the singular locus contains the Segre.
Choose any component of the singular locus and a hyperplane $x_I=0$ that does not contain it. The intersection of the secant with the complement of this hyperplane is isomorphic to the trivial bundle over $T_{2,n}$. The singular component must be one of the singular components described in Proposition~\ref{prop:codimsing}. In particular, it must contain the fiber of the vector bundle over the zero point, which
corresponds to the Segre variety.

Since the Segre is not contained in any of the hyperplanes $x_I=0$, also all the components of the singular locus have nonempty intersection with the complement of any hyperplane. In particular each component of the singular locus corresponds to a component of the singular locus of $T_{2,n}$.

We will prove that the intersection of any two components is precisely the Segre variety.
 Fix an affine that is a complement of a hyperplane $x_I=0$. Two components, in cumulant coordinates, correspond to trivial vector bundles over two distinct lines. Thus their intersection is the trivial vector bundle over $0$ that is precisely the Segre.

Now let us describe the singular locus explicitly in the original coordinates. Choose two indices $i_1,i_2\in [n]$. We give a description of the component of the singular locus corresponding to the vertex $v$ with exactly two nonzero entries on coordinates $i_1$ and $i_2$. Note that the intersection of the facets $q_i=0$ for $i\neq i_1,i_2$ is precisely the vertex $v$. This means that the intersection of the toric divisors corresponding to these facets is precisely the given component of the singular locus. Let us parameterize this intersection. By
Proposition~\ref{prop:toricdivisors} the restriction of the parameterization map:
\[(t,a_1,\dots,a_n,b_1,\dots,b_n)\mapsto \left(x_{I} = t\prod_{i\in I}a_i+(1-t)\prod_{i\in I}b_i\right) \]
to the subspace $a_i=b_i$ for $i\neq i_1,i_2$ is contained in the intersection of the given divisors. Although there are $n+3$ parameters, the closure of the image of the restriction is not of dimension $n+3$. Indeed, looking at the toric variety we see that the given restriction corresponds to setting all $d_i=0$ for $i\neq i_1,i_2$. Thus
the dimension of the closure of the image equals $n+1$ and coincides with the affine bundle over the singular line.

We can also see that each component of the singular locus of the secant is isomorphic to $\P^3\times (\P^1)^{\times (n-2)}$. Indeed, consider the Segre embedding of this variety:
\[((a_0',a_1',a_2',a_3'),(b_1^1,b_2^1),\dots,(b^{n-2}_1,b^{n-2}_2))\mapsto (a_i'\prod_{k}b_j^k),\]
where $0\leq i\leq 3$, $j=1,2$. We may restrict the parameterization to $a_0'=b_1^k=1$ for all $k$. Recall that the secant of $\P^1\times \P^1$ fills the whole $\P^3$. In particular, we see that the image of the map coincides with the parameterization of the singular component corresponding to the subset $\{i_1,i_2\}$, where we set $b_2^k=b_k(=a_k)$ for $k\neq i_1,i_2$ and $(a_i')\in \P^3$ is the point corresponding to the point on the secant line $t((1,a_{i_1})\times(1,a_{i_1}))+(1-t)((1,b_{i_1})\times(1,b_{i_1}))$.
\kdow

%\piotr{\begin{rem}In algebraic statistics the variety $V$ corresponds to the naive Bayes model. In this context the singular locus was described in \cite[Theorems 13 and 14]{GHKM}. Note that their coordinates do not correspond exactly to our secant cumulants.
%\end{rem}
%}
%\mateusz{We can add this remarks, but we should be careful - they work with real parameters from $0$ to $1$. A priori the singular locus (?) could be very different.}

The toric description allows one to find a precise description of the singularities of the secant.  Specifically, we have the following characterization.
\begin{cor}\label{cor:singp1}
For $n=2$ and $n=3$ the secant is a smooth variety. For $n\geq 4$ it has rational singularities, but is never $\Q$-factorial. There are $n\choose 2$ components of the singular locus, each isomorphic to the secant of $\P^3\times (\P^1)^{\times (n-2)}$. Each two components intersect precisely in the Segre.
For $n=5$ the variety $T_{2,n}$ is Gorenstein, with terminal singularities.
For $n\geq 4$, $n\neq 5$ the variety is not $\Q$-Gorenstein.
\end{cor}
\dow
As we have shown in Proposition~\ref{prop:codimsing} for $n\geq 4$ we can find a cone in the fan that is not simplicial, thus the variety is not $\Q$-factorial.

Consider the cone $\sigma\in \R^{n+1}$ over the polytope $P_{2,n}$ corresponding to the affine toric variety $T_{2,n}$. Assume $n\geq 4$. By Lemma~\ref{prop:fan} the ray generators of the dual cone $\sigma^\vee$ are:
\begin{itemize}
\item $e_i$ for $i=1,\dots, n$;
\item $e_0-e_i$ for $i=1,\dots, n$;
\item $-2e_0+\sum_{i=1}^n e_i$.
\end{itemize}
Assume $n=5$. The ray generators belong to the hyperplane $2e_0^\vee+\sum_{i=1}^4 e_i^\vee=1$. This proves that the variety $T_{2,n}$ is Gorenstein -- \cite[Definition 8.2.14]{Cox}. As there are no integral points in the convex hull of the ray generators and the point zero, by \cite[Proposition 11.4.12]{Cox} we see that the singularities are terminal.

For $n\geq 4$, $n\neq 5$ it is a straightforward check that the ray generators do not belong to an affine subspace, thus the variety is not $\Q$-Gorenstein.
\kdow
\section{Relation to flattenings}\label{sec:flats}

Given a partition $I_{1}|I_{2}$ of an index $I$, let $M^{x}_{I_{1}|I_{2}}$ denote the flattening of $x = [x_{I}]$ with row indices given by subsets $A\subseteq I_{1}$ and column indices given by subsets $B\subseteq I_{2}$ and the entry indexed by $A,B$ is $x_{A\cup B}$. For example the flattening of $x$ corresponding to a partition $12|34$ is given by
\[
M_{12|34}^x:=
\left(\begin{array}{cccc}
x_\emptyset & x_3     & x_4     & x_{34}\\
x_1         & x_{13}  & x_{14}  & x_{134}\\
x_2         & x_{23}  & x_{24}  & x_{234}\\
x_{12}      & x_{123} & x_{124} & x_{1234}
\end{array}\right).
\]
Similarly define $M^{y}_{I_{1}|I_{2}}$ and $M^{z}_{I_{1}|I_{2}}$ as flattenings of $y = [y_{I}]$ and $z=[z_{I}]$ respectively.

 Recently it was proved by Raicu \cite{RaicuGSS} that $3\times 3$  minors of all flattenings of the tensor $x=[x_I]$ generate the ideal of the secant variety. In this section we show how these equations are related to the generators of $\cI(V)$. In Theorem~\ref{thm:flat} and Corollary~\ref{cor:minorsles3} we present a basic proof that if we consider flattenings with $I_1$ of cardinality at most $3$, then the $3\times 3$  minors define the projective scheme of the secant variety.

In what follows we will need another description of the generators of the ideal of $T_{a,b}$.
\begin{defn}[Flattening quadrics]\label{flatquad}
Choose any partition $I_1| I_2$ of $[n]$. Consider the following \textit{restricted flattening matrix} of $z=[z_I]$ with respect to this partition:
\begin{itemize}
\item rows are indexed by \emph{nonempty} subsets of $I_1$;
\item columns are indexed by \emph{nonempty} subsets of $I_2$;
\item an entry indexed by subsets $A\subseteq I_1$, $B\subseteq I_2$ is zero, if $|A\cup B|< a$\ or $|A\cup B|>b$. Otherwise it is the variable $z_{A\cup B}$.
\end{itemize}
For example, for $T_{2,3}^{5}$ we associate to the partition $15|234$ the restricted flattening matrix:
\[
\begin{pmatrix}
 &\vline& 2 & 3 & 4 & 23 & 24 & 34 & 234\\
\hline
1& \vline& z_{12} & z_{13} &z_{14} & z_{123} &z_{124} & z_{134} & 0 \\
5& \vline& z_{25} & z_{35} &z_{45} & z_{235} &z_{245} & z_{345} & 0 \\
15&\vline& z_{125} & z_{135} &z_{145} & 0 &0 & 0&0
\end{pmatrix}
.\]

We define the \emph{flattening quadrics} as $2\times 2$ minors of this matrix involving only nonzero entries. Every such a quadratic relation lies in the ideal of $T_{a,b}$. So, in the above example, we would consider
$\left|\begin{matrix} z_{123} & z_{134}\\
z_{235} & z_{345}
\end{matrix}\right|$,
but not
$\left|\begin{matrix} z_{35} & z_{235}\\
 z_{135} & 0
\end{matrix}\right|$.
\end{defn}

\begin{lem}\label{flats}
The flattening quadrics described in Definition~\ref{flatquad} generate the ideal of $T_{a,b}$ for $a\geq 2$. Moreover it is enough to consider partitions $I_1|I_2$ for which the set $I_1$ is of cardinality at most $3$.
\end{lem}
\dow
By Proposition~\ref{prop:Tabideal} it is enough to prove that the flattening quadrics generate \emph{bumping} and \emph{swapping}.

1) Swapping. Suppose that $a-1\leq |I|,|J|\leq b-1$ and $i,j\notin I\cup J$. Consider the partition $\{i,j\}| ([n]\setminus\{i,j\})$. Then swapping is given by the equation given by the $2\times 2$ minor of the matrix
\[
\begin{pmatrix}
z_{I\cup \{i\}} & z_{J\cup \{i\}} \\
z_{I \cup \{j\}}& z_{J\cup\{j\} }\\
\end{pmatrix}
\]
with rows indexed by  $\{i\},\{j\}$ and columns by $I$, $J$.

2) Bumping. Suppose $a\leq |I|,|J|\leq b-1$ and $i\notin I\cup J$. Consider two cases. First, if there exists $j\in I\cap J$ then consider the partition $\{i,j\}| ([n]\setminus\{i,j\})$ and  the minor given by
\[
\begin{pmatrix}
z_{I\cup \{i\}} & z_{I}\\
z_{J\cup\{i\}} & z_{J}\\
\end{pmatrix}
\]
with rows indexed by $I\setminus \{j\}$, $J\setminus \{j\}$ and columns $\{i,j\}$, $\{j\}$. Second, if $I\cap J=\emptyset$, choose two indices $j\in I\setminus J$ and $k\in J\setminus I$. Consider the partition $\{i,j,k\}| ([n]\setminus\{i,j,k\})$ and the minor given by
\[
\begin{pmatrix}
z_{I\cup \{i\}} & z_{(J\setminus\{k\} \cup \{j\})\cup \{i\}} \\
z_{I\setminus\{j\}\cup \{k\}}& z_{J}\\
\end{pmatrix}
\]
with rows indexed by  $\{i,j\}$, $\{k\}$  and columns by $I\setminus \{j\}$, $J\setminus \{k\}$. This is not yet a bumping relation. However by swapping we obtain
\[
z_{I\setminus\{j\}\cup \{k\}}z_{(J\setminus\{k\} \cup \{j\})\cup \{i\}}\quad=\quad z_{I}z_{J\cup \{i\}}.
\]
Hence, up to swapping of elements $j,k$, the above minor generates the bumping relation.
\kdow

From Lemma~\ref{flats} it follows that $V$ is described by the flattening quadrics in $z=[z_I]$. In what follows we show how these flattening quadrics are related to $3\times 3$ minors of the flattenings of the tensor $x=[x_I]$.

Let us first present an example for $n=4$. We start with $M^{x}_{12|34}$ (described above) and we assume that $x_\emptyset=1$. Perform obvious elementary operations to set to zero the off-diagonal elements of the first column (e.g. subtract from the second row the first row multiplied by $x_1$). This results in the following matrix
\[
\left(\begin{array}{cccc}
1 & x_3 & x_4 & x_{34}\\
0 & x_{13}-x_{1}x_{3} & x_{14}-x_{1}x_{4} & x_{134}-x_{1}x_{34}\\
0 & x_{23}-x_{2}x_{3} & x_{24}-x_{2}x_{4} &  x_{234}-x_{2}x_{34}\\
0 & x_{123}-x_{12}x_{3} & x_{124}-x_{12}x_{4} & x_{1234}-x_{12}x_{34}
\end{array}\right).
\]
Using simple elementary operations we also reduce elements in the first row to zeros.
Note that $z_{ij}=x_{ij}-x_ix_j$ for all $i<j$. Now subtract from the last column: the second column multiplied by $x_4$ and the third column multiplied by $x_3$. Similarly subtract from the last row: the second row multiplied by $x_2$ and the third row multiplied by $x_1$. For every $1\leq i<j<k\leq 4$
\[
z_{ijk}=x_{ijk}-x_ix_{jk}-x_jx_{ik}-x_kx_{ij}+2x_ix_jx_k,
\]
and
\begin{equation}\label{eq:z1234}
\begin{array}{rcl}
z_{1234}&=&x_{1234}-x_{12} x_{34}-x_1 x_{234}-x_2 x_{134}-x_3 x_{124}-x_4 x_{123}+\\
&+&2 x_{12}x_3 x_4 +x_{13}x_2 x_4 +x_{14}x_2 x_3 +x_1 x_{23}x_4 +x_1 x_{24}x_3 +2 x_1 x_2 x_{34}-4 x_1 x_2 x_3 x_4.
\end{array}
\end{equation}

We check directly that every element of the resulting matrix is equal to the corresponding $z_I$. Hence, after all the above elementary operations we obtain
\[
\widehat M_{12|34}^{z}:=\left(\begin{array}{cccc}
1 & 0 & 0 & 0\\
0 & z_{13} & z_{14} & z_{134}\\
0 & z_{23} & z_{24} & z_{234}\\
0 & z_{123} & z_{124} & z_{1234}
\end{array}\right).
\]
The rank of this matrix is one higher than the rank of
\[
(M_{12|34}^{z})^{(1)}:=\left(\begin{array}{ccc}
 z_{13} & z_{14} & z_{134}\\
 z_{23} & z_{24} & z_{234}\\
 z_{123} & z_{124} & z_{1234}
\end{array}\right).
\]
In particular, as the matrices were obtained by row and column operations, the $2\times 2$ minors of $(M_{12|34}^{z})^{(1)}$ are linear combinations with polynomial coefficients of the $3\times 3$ minors of ${M}_{12|34}^x$.

The same is not true if we consider the flattening matrix $M_{13|24}^x$. After applying the same elementary operations as to $M_{12|34}^x$ the right-bottom element of the resulting matrix will not be equal to $z_{1234}$. One way to see that is to note that the formula for $z_{1234}$ given in (\ref{eq:z1234}) contains a term $-x_{12}x_{34}$ but it does not contain $-x_{13}x_{24}$ and so the symmetry is broken.

If $M$ and $M'$ are matrices similar by elementary row and column operations we denote this by $M\sim M'$. For a matrix $M$, let $\widehat M$ denote the matrix obtained by replacing the first row and column with all entries equal to $0$, apart from the $(1,1)$ entry equal to $1$, and let $M^{(i)}$ denote the principal sub matrix of $M$ obtained by deleting the $i$-th row and column.
\begin{lem}\label{lem:flatinz}
Let $A_1|A_2$ be an interval partition of $[n]$.
On $U_{\emptyset}$ we have
\[M^x_{A_1|A_2}\,\,\,\sim\,\,\, \widehat M^{z}_{A_{1}|A_{2}}.\]
In particular, the $2\times 2$ minors of $({M}^z_{A_1|A_2})^{(1)}$ are generated, as polynomials, by
$3\times 3$ minors of $M^x_{A_1|A_2}$.
\end{lem}
\begin{proof}
Let $I\subseteq A_1$ and $J\subseteq A_2$. Define $\widetilde M^y_{A_1|A_2}$ to be equal to  $M^y_{A_1|A_2}$ everywhere apart from these elements of the first row and column that correspond to $(I,J)=(\{i\},\emptyset)$ and $(\emptyset,\{i\})$, which are assumed to be zero. Let $\tilde y$ be the entries of the matrix $\widetilde M^y$. Note that the formula (\ref{eq:cremonaxtoy}) holds for \emph{all} $\tilde y_I$.
As $A_1$ and $A_2$ are disjoint we have:
\[
\tilde y_{I\cup J}=\sum_{I'\subseteq I}\sum_{J'\subseteq J}(-1)^{|I\setminus I'|}(-1)^{|J\setminus J'|}x_{I'\cup J'}\prod_{i\in I\setminus I'}x_i\prod_{i\in J\setminus J'}x_i.
\]
We can rearrange that to obtain
\[
\tilde y_{I\cup J}=\sum_{I'\subseteq I}\sum_{J'\subseteq J}u_{II'}x_{I'\cup J'}v_{J'J},
\]
where $u_{II'}=(-1)^{|I\setminus I'|}\prod_{i\in I\setminus I'}x_i$ and $v_{J'J}=(-1)^{|J\setminus J'|}\prod_{i\in J\setminus J'}x_i$. Therefore, we have $$\widetilde M^y_{A_1|A_2}=U M_{A_1|A_2}^x V,$$ where $U=[u_{IJ}]$ is lower triangular (in any total ordering of the standard basis such that $e_I<e_J$ if $I\subset J$) and $V=[v_{IJ}]$ is upper triangular and both have ones on the diagonal.

From now on ${\rm IP}({[n]})$ denotes the poset of interval partitions \emph{with no singleton blocks}. Let $A_1|A_2\in {\rm IP}({[n]})$.
Let us consider the following triangular matrices $\widetilde{U}=[\widetilde{u}_{IJ}]$, $\widetilde{V}=[\widetilde{v}_{IJ}]$ with ones on the diagonal. Consider two subsets $I,I'\subseteq A_1$. We define:
$$\widetilde{u}_{II'}:=\sum_{\beta\in {\rm IP}({I\setminus I'})} (-1)^{|\beta|-1}\prod_{B\in \beta(I\setminus I')}y_B$$
if $I'\subset I$ and  all elements of $I'$ are greater than all elements of $I\setminus I'$ -- notice that $I'$ may be empty. Otherwise we set $\widetilde{u}_{II'}:=0$.
Analogously we define:
$$\widetilde{v}_{J'J}:=
\sum_{\gamma\in {\rm IP}({J\setminus J'})} (-1)^{|\gamma|-1}\prod_{B\in \gamma(J\setminus J')}y_B$$
if $J'\subset J$ and all elements of $J'$ are smaller than all elements of $J\setminus J'$ and $\widetilde{v}_{J'J}:=0$ otherwise.
We claim that
$$\widehat M^{z}_{A_{1}|A_{2}}=\widetilde{U}\widetilde{M}^y_{A_1|A_2}\widetilde{V}.$$
Indeed, let us consider an entry indexed by subsets $I, J$. If both of them are not empty we have:
\[
z_{I\cup J}=\sum_{\pi\in {\rm IP}({I\cup J})}(-1)^{|\pi|-1}\prod_{B\in \pi}y_B.
\]
The corresponding $(I,J)$-th entry of the matrix $\widetilde{U}\widetilde M^y_{A_1|A_2}\widetilde{V}$ equals:
$$\sum_{I'\subseteq A_1}\sum_{J'\subseteq A_1}u_{I I'}\tilde y_{I'J'}v_{J'J}=\sum_{I'\subseteq I}\sum_{J'\subseteq J}u_{I I'}\tilde y_{I'J'}v_{J'J}$$
$$=\sum_{I'\subseteq I}\sum_{J'\subseteq J}(\sum_{\beta\in {\rm IP}({J\setminus J'})} (-1)^{|\beta|-1}\prod_{B\in \beta(J\setminus J')}y_B)y_{I'J'}(\sum_{\gamma\in {\rm IP}({I\setminus I'})} (-1)^{|\gamma|-1}\prod_{B\in \gamma(I\setminus I')}y_B).$$

Let us compare this with the expression for $z_{I\cup J}$ above. First of all notice that interval partitions $\pi$ in the expression for $z_{I\cup J}$ that are not subdivisions of $I|J$ correspond exactly to those elements in the above double sum, where $I'$ and $J'$ are both nonempty. Indeed it is enough to choose $I'\cup J'$ to be the only partition in $\pi$ that has got nonempty intersection with $I$ and $J$. We choose $\beta$ and $\gamma$ to be compatible with $\pi$.
Now to each partition $\pi$ that is a subdivision of $I|J$ let us associate exactly $3$ elements of the sum indexed by $I',J',\beta,\gamma$. For the first element we take $I'=J'=\emptyset$ and compatible $\beta$ and $\gamma$. Notice that this element appears with opposite sign than in $z_{I\cup J}$. The following two elements are:
\begin{enumerate}
\item $I'=\emptyset$, $J'$ equal to the first part in $J$;
\item $J'=\emptyset$, $I'$ equal to the last part in $I$.
\end{enumerate}
These two elements appear with the same sign as in $z_{I\cup J}$. Thus the sum is indeed equal.

Suppose now that $I=\emptyset$ and $J$ is nonempty. In this case we must have $I'=\emptyset$. We see that, as above, we can pair the elements appearing in the sum by taking $J'=\emptyset$ or $J'$ equal to the first partition in $J$. These two elements appear with the opposite sign, so indeed the entries in both matrices are equal to zero -- in $\widehat M^z$ just by definition. The case $J=\emptyset$ and $I$ nonempty is analogous.

If $I=J=\emptyset$ the entries of both matrices are equal to $1$.

As $\widetilde U$ and $\widetilde V$ are triangular we obtain the proof of the lemma.
\end{proof}

Let us now consider all the partitions $[n]=I_1| I_2$ and the corresponding flattening matrix.
\begin{thm}\label{thm:flat}
The $3\times 3$  minors of all flattening matrices  $M^{x}_{I_{1}|I_{2}}$ generate the ideal of $V$.
\end{thm}
\dow
Consider a filtration on the rings $A_2\subset\dots\subset A_n=\C[z_I:\, I\subseteq [n]]$, where $A_i$ is the polynomial ring in coordinates $z_I$ with $I\neq\emptyset$ of cardinality at most $i$. Let $J'$ be the ideal generated by the $3\times 3$  flattenings in the ring $A_n$ (we put $x_\emptyset=1$). We want to show that $\cI(V)$ is generated by $3\times 3$ minors of all flattening matrices $M^{x}_{I_{1}|I_{2}}$.

We define the ideal $J_i$ in the ring $A_i$ as $\cI(V)\cap A_i$, which is a toric ideal. Let us note that any polynomial $f$, in any toric ideal, can be expressed as a sum of binomials that are only in variables appearing in $f$. It follows that  $J_i$ is the ideal of the variety $T_{2,i}^n$. Thus, by Lemma~\ref{flats}, $J_i$ is generated by $2\times 2$ minors of flattening quadrics in the sense of Def.~\ref{flatquad} (not indexed by the empty set) in coordinates $z_I$, with $|I|\leq i$. We also define the ideal $J'_i$ of the ring $A_i$ as the ideal generated by those $3\times 3$  minors of the flattening matrices in coordinates $x_I$ for which $|I|\leq i$. We will inductively prove that $J_i\subseteq J'_i$ -- the other inclusion is obvious.

Step 1. For $i=2$ consider any $2\times 2$ minor that is a flattening quadric (with columns and rows not labeled by the empty set). It is equal to the corresponding $3\times 3$  minor of the flattening matrix in $x_I$ where we add the row and column labeled by the empty set:
\[
\begin{pmatrix}
1 & x_k & x_l\\
x_m & x_{mk} & x_{ml}\\
x_j & x_{jk} & x_{jl}
\end{pmatrix}.
\]
 To show that, it suffices to check that $z_{mk}z_{jl}-z_{ml}z_{jk}$ is the determinant of the above matrix, which follows immediately from the fact that $z_{ab}=x_{ab}-x_ax_b$.

Induction: Consider any minor $z_{B\cup C}z_{D\cup E}-z_{B\cup E}z_{D\cup C}$ with nonempty $B,D\subseteq I_1$, $C,E\subseteq I_2$ and $|B\cup C|, |D\cup E|, |B\cup E|, |D\cup C|\leq i+1$. Let us choose a bijection $\sigma$ of the set $[n]$ such that $\sigma(I_1)|\sigma(I_2)$ is an interval partition. Define the variables $z'_I$ by the same formula, as $z_I$ but with the ordering of the set $[n]$ induced from $\sigma$ (see also Remark~\ref{rem:order}). The formula in Lemma~\ref{lem:parametryzacja} holds for $z'_I$ as it is independent on the ordering. In particular $z_I-z'_I\in J_{i+1}$ if $|I|\leq i+1$. Moreover, we have $z_I=x_I+g(x)$ and $z'_I=x_I+h(x)$, where $g$ and $h$ are polynomials in variables labeled by sets of cardinality strictly smaller than $|I|$. In particular $z_I-z'_I\in A_i$, thus $z_I-z'_I\in J_i$. Hence, by induction $z_I-z'_I\in J'_i\subset J'_{i+1}$.
In order to prove that the given minor is in $J'_{i+1}$ it is enough to consider the same minor with all $z$ replaced by $z'$. Notice that the minor in $z'$ belongs to $J'_{i+1}$ by Lemma~\ref{lem:flatinz}.
\kdow
\begin{cor}\label{cor:minorsles3}
The $3\times 3$ minors of the flattening matrices corresponding to partitions $I_1| I_2$ with $|I_1|\leq 3$ define the projective scheme of the secant variety.
\end{cor}
\dow
From Theorem~\ref{thm:flat} we know that the flattenings define the scheme on the open affine $x_\emptyset\neq 0$. Obviously, we obtain the same results on any open affine set $U_I$, by repeating the same proof, letting $x_I$ play the role of $x_\emptyset$. These open sets cover the whole space.
\kdow

\section{Generalization to higher dimensions}\label{sec:generalizations}

It is well-known that, in order to describe the ideal of the $k-th$ secant of the Segre variety when all the vector spaces have dimension at least $k$, it is enough to consider vector spaces of dimension $k$ \cite[Prop.~5.1]{LandsbergWeyman}. Still, it is not clear what other properties of the secant variety are inherited from vector spaces of dimension $k$ to higher dimensions. One might expect that normality, Cohen-Macaulay and rational singularities are inherited. However it is a nontrivial result \cite[Lemma 5.3]{LandsbergWeyman}, with an additional technical assumption, that the property of being arithmetically Cohen-Macaulay is inherited.
The aim of this section is to show that all the arguments in our paper can be adopted to higher dimensional vector spaces.

\subsection{Secants in higher dimensions}

Fix vector spaces $V_1,\dots,V_n$ where $\dim V_i=k_i+1$. Fix a basis $e_0^i,\dots, e_{k_i}^i$ of $V_i$. The basis of the tensor product $V_1\otimes\dots\otimes V_n=:U$ can be identified with sequences $(i_1,\dots,i_n)$ with $0\leq i_j\leq k_i$. Slightly abusing notation we denote these sequences by $I$, where $I\in \prod_j \{0,\ldots,k_j\}$. Let $x_{I}=x_{(i_1,\dots,i_n)}$ be the homogeneous coordinates of the projectivization of the tensor product. Further, let $Supp(I)$ denote the support of $I$, that is $Supp(I)=\{j\in [n]:\, i_j\neq 0\}$, and $|I|$ to be the cardinality of $Supp(I)$.
We change the coordinates of the affine space $U_{{(0,\dots,0)}}$
as follows:
\[y_{I}:=x_{I} \quad \mbox{for} \quad |I|=1,\]
\[y_{I}:=y_{(i_1,\dots,i_n)}=\sum_{A\subseteq Supp(I)}(-1)^{|I\setminus A|}x_{\{i_a:a\in A\}}\prod_{j\in Supp(I)\setminus A}x_{i_j}, \quad\mbox{for}\quad |I|\geq 2,\]
where we use the convention that an index missing in a subscript of a variable is always equal to $0$. With the same convention:
\[z_{I}:=y_{I} \quad \mbox{for} \quad |I|=1,\]
$$z_I:=z_{(i_1,\dots,i_n)}=\sum_{\pi\in {\rm IP}(Supp(I))}(-1)^{|\pi|-1}\prod_{B\in \pi}y_{\{i_b:b\in B\}},\qquad\mbox{for }|I|\geq 2$$
where the last sum is taken over interval partitions of the set $I$ without singleton blocks.
Note that this is exactly the same change of coordinates that we have used before (in \eqref{eq:cremonaxtoy} and \eqref{eq:cremonaytoz}), as in each vector space $V_k$ only two basis vectors are involved: $e_0^k$ and $e_{i_k}^k$.

Consider the Segre embedding given by

\begin{eqnarray*}
\P(V_{1})\times \dots \times \P(V_{n})&\to& \P(V_{1}\otimes\dots\otimes V_{n})\\
\left([a_{0}^{1},\ldots,a^{1}_{k_1}],\dots,[a^{n}_{0},\ldots,a^{n}_{k_n}] \right)&\mapsto&
\left[
x_{(i_1,\dots,i_n)}\;=\; \prod_{j=1}^n a^j_{i_j}
\right]
.\end{eqnarray*}
On an affine open set $U_{{(0,\dots,0)}}$ we can assume that $[a^{i}_{0},\ldots,a^{i}_{k_i}] = [1,a^{i}_1,\ldots,a^i_{k_i}]$ and the Segre embedding is parameterized by
\[
x_I\quad=\quad \prod_{j\in Supp(I)}a^j_{i_j},\qquad\mbox{for all }I\in \prod_j \{0,\ldots,k_j\}.
\]
 On this open subset the secant variety ${\rm Sec}(\P(V_{1})\times \dots \times \P(V_{n}))$ is parameterized by
\[
x_I\quad=\quad x_{(i_1,\dots,i_n)}\quad=\quad (1-t)\prod_{j\in Supp(I)} a^j_{i_j}+t\prod_{j\in Supp(I)} b^j_{i_j},\qquad\mbox{for all } I\in \prod_j\{0,\ldots,k_j\}
,\]
where $a^{j}_{i_{j}}$ and $b^{j}_{i_{j}}$ are $\C$ valued parameters. We introduce the affine variety $V$ given by
\[
V\quad:=\quad {\rm Sec}(\P(V_{1})\times \dots \times \P(V_{n})) \cap U_{{(0,\dots,0)}}.
\]
Since $Supp(I)$ is a subset of $[n]$, we can literally reprove Lemma~\ref{lem:parametryzacja} obtaining the following.
\begin{lem}\label{lem:parametryzacja2}
The variety $V$ in the coordinate system given by the higher order secant cumulants is the Zariski closure of the image of the parameterization given by:
\[
z_{(0,\dots,0,i_j,0,\dots,0)}=(1-t)a_{i_j}^j+tb_{i_j}^j,\]
\[
z_I=z_{(i_1,\dots,i_n)}= t(1-t)(1-2t)^{|I|-2}\prod_{j\in Supp(I)}(b_{i_j}^j-a_{i_j}^j)\text{ for }|I|\geq 2.
\]
\kwadrat
\end{lem}

Consider the lattice $M:=\Z^{k_1+1}\times\dots\times\Z^{k_n+1}$, where each factor has a distinguished basis.
\begin{defn}[Polytope $P$]
To every $I=(i_1,\ldots,i_n)\in \prod_j\{0,\ldots,k_j\}$ such that $|I|\geq 2$ we associate a point $e_{i_1}^1\oplus\cdots\oplus e_{i_n}^n$ in $M$, where $e^i_j$ are corresponding unit vectors. Denote by $\J$ the set of all these points. Define now $P$ as the convex hull of $\J$.
\end{defn}
This is a direct generalization of the set $\J_{2,n}$ and polytope $P_{2,n}$ introduced in Definition~\ref{J}. Directly from the definition it follows that integral points of $P$ satisfy:
\begin{itemize}
\item each point belongs to the cube $\prod_{j=1}^n[0,1]^{k_j+1}$;
\item after projecting to any $\Z^{k_j+1}$ exactly one coordinate is nonzero;
\item there exist at least two projections to $\Z^{k_{j_1}+1}$ and $\Z^{k_{j_2}+1}$ such that the  first coordinate is zero.
\end{itemize}

Different versions of such polytopes have already appeared in the literature. Without the last assumption we obtain the usual Segre. Also similar varieties appear in phylogenetics, where basis vectors of each lattice correspond to group elements. This is the case of group--based models \cite{SturmfelsSullivant08}, $G$-models \cite{mateuszjalg} and certain toric varieties arising from graphs \cite{Buczynska}.

Denote by $T$ the toric variety associated to the polytope $P$.
\begin{thm}
The variety $V$ is a toric variety that is the trivial affine bundle of rank $k_1+\cdots +k_n$ over the spectrum of the algebra associated to the monoid generated by integral points of $T$.\kwadrat
\end{thm}
It is not obvious that $P$ is normal. Let us prove more, namely that it admits a unimodular triangulation. As before in Section~\ref{sec:secant}, we find a squarefree Gr\"obner basis. First let us describe the monomial order we use.

\begin{defn}[Order of the variables]
We have $z_J< z_I$ if either
\begin{itemize}
\item $|Supp(J)|< |Supp(I)|$ or
\item $|Supp(I)|=|Supp(J)|$ and $\min(Supp(J)\setminus Supp(I))<\min(Supp(I)\setminus Supp(J))$ or
\item $Supp(I)=Supp(J)$ and $J<I$ lexicographically.
\end{itemize}
\end{defn}
\begin{thm}\label{thm:GB}
The following set of binomials forms a square free Gr\"obner basis, with respect to reverse lexicographic order, of the ideal of the toric variety given by $P$:
\begin{enumerate}
\item $z_{i_1,\dots,i_n}z_{j_1,\dots,j_a=0,\dots,j_n}-z_{i_1,\dots,j_{a}=0,\dots,i_n}z_{j_1,\dots,i_a,\dots,j_n}$
    where the support $|i_1,\dots,i_n| \geq 3$ (cf. bumping),
\item $z_{i_1,\dots,i_a\neq 0,\dots,i_n}z_{j_1,\dots,j_a\neq 0,\dots,j_n}-z_{i_1,\dots,j_a,\dots,i_n}z_{j_1,\dots,i_a,\dots,j_n}$ (pseudoswapping),
\item $z_{i_1,\dots,i_a=0,\dots,i_b\neq 0,\dots,i_n}z_{j_1,\dots,j_a\neq 0,\dots,j_b=0,\dots,j_n}-z_{i_1,\dots,j_a,\dots,j_b=0,\dots,i_n}z_{j_1,\dots,i_a= 0,\dots,i_b,\dots,j_n}$ (swapping),
\item $z_{i_1,\dots,i_a=0,\dots,i_b\neq 0,\dots i_n}z_{j_1,\dots,j_a,\dots,j_b=0,\dots,j_n}z_{l_1,\dots,l_a\neq 0,\dots,l_b\neq 0,\dots ,l_n}-\\
    z_{i_1,\dots,l_a,\dots,j_b= 0,\dots ,i_n}z_{j_1,\dots,j_a,\dots,i_b\neq 0,\dots,j_n}z_{l_1,\dots,i_a= 0,\dots,l_b\neq 0,\dots, l_n}$,
    where the support $|i_1,\dots,i_n| \geq 3$
     and $|i_1,\dots,i_a=0,\dots,i_b\neq 0,\dots ,i_n|=2$,
\item $z_{i_1,\dots,i_a=0,\dots,i_b\neq 0,\dots, i_n}z_{j_1,\dots,j_r\neq 0,\dots,j_b=0,\dots,j_n}z_{l_1,\dots,l_a\neq 0,\dots,l_r= 0,\dots, l_n}-\\
    z_{i_1,\dots,l_a,\dots,j_b= 0,\dots ,i_n}z_{j_1,\dots,l_r=0,\dots,i_b\neq 0,\dots,j_n}z_{l_1,\dots,i_a= 0,\dots,j_r\neq 0,\dots ,l_n}$, where the supports are $|(l_1,\dots,l_n)|=|(i_1,\dots,i_a=0,\dots,i_b\neq 0,\dots, i_n)|=2$, however we do \emph{not} assume that $r<b$.
\end{enumerate}
\end{thm}
\dow
Choose any binomial $m_1-m_2$ in the ideal with $m_1>m_2$. Let $z_{i_1,\dots,i_n}=:z_I$ be the smallest variable dividing $m_1$ and let $z_{j_1,\dots,j_n}=:z_J$ be the smallest variable dividing $m_2$. Let $I_s:=Supp(I)$ be the support of $z_I$ and $J_s:=Supp(J)$ of $z_J$.
Without loss we may assume that $z_J<z_I$.

Consider the following two cases:

1) $|I|>2$. If there is a variable $z_K|m_1$ such that $I_s\not\subseteq {\rm Supp}(K)$, then we may apply relation (1) reducing $I_s$. Thus we assume that $I_s$ is contained in the support of every variable dividing $m_1$. Hence the same is true for all variables dividing $m_2$, in particular, for $z_J$. As $z_J<z_I$, we must have $I_s=J_s$ and $i_a>j_a$ for certain $a\in I_s$. So there must be a variable $z_L:=z_{l_1,\dots,l_n}|m_1$ with $l_a=j_a$. We can apply relation (2) to $z_Lz_I$.

2) $|I|=2$. Proceeding as in step 1), we know that $I_s\neq J_s$. Let $i:=max(I_s\setminus J_s)$. There exists $z_L|m_1$ such that $i\notin {\rm Supp}(L)$. In particular, as $|I|=2$ and $|L|\geq 2$, we know that ${\rm Supp}(L)\not\subseteq I_s$. Thus we can define $j$  as the smallest index $j\not \in I_s$ appearing in the support of some $z_K|m_1$. We consider three subcases.

2a) $j<min(I_s)$. If $I_s\not\subseteq {\rm Supp}(K)$ we may apply relation $(3)$. Otherwise $|K|\geq 3$ and we may apply relation $(4)$.

2b) $min(I_s)<j<max(I_s)$. First note that  $min(J_s)=min(I_s)$. Indeed, we know that $min(J_s)\leq min(I_s)$ and by the choice of $j$ the equality cannot be strict. Thus $i=max(I_s\setminus J_s)=max(I_s)$. Recall that $i\notin{\rm Supp}(L)$. Notice that $i\in {\rm Supp}(K)$ because otherwise we could apply relation $(3)$. If $|K|\geq 3$ then apply relation (4). If $|K|=2$ then there exists $r\in {\rm Supp}(L)\setminus {\rm Supp}(K)$. Apply relation (5).

2c) $max(I_s)<j$. In this case we must have $I_s=J_s$, hence apply relation (2), as in case 1).\kdow

The Gr\"obner basis in Theorem~\ref{thm:GB} is square-free with respect to a degrevlex order. By \cite[Corollary 8.9]{Sturmfels96} it induces a unimodular pulling triangulation. Let us recall the recursive construction of the pulling triangulation -cf. \cite[p. 67]{Sturmfels96}.

Consider any polytope with  linearly ordered integral points. If it is a simplex there is nothing to be done.
\begin{enumerate}
\item Choose the smallest integral point $p$ of the polytope;
\item Consider the set $\mathcal{F}$ of facets that do not contain $p$;
\item For each element of $\mathcal{F}$ perform the pulling triangulation;
\item Enlarge each simplex on facets by $p$.
\end{enumerate}

\begin{exmp}
Consider the variety $T_{2,4}$ for $n=4$. The associated polytope is of dimension $4$, has volume equal to $1/2$ and $11$ (ordered) vertices in $\{1\}\times\Z^4$
\[(1,1,1,0,0);(1,1,0,1,0);(1,1,0,0,1);(1,0,1,1,0);(1,0,1,0,1);(1,0,0,1,1);\]
\[(1,1,1,1,0);(1,1,1,0,1);(1,1,0,1,1);(1,0,1,1,1);(1,1,1,1,1).\]
The pulling triangulation consists of 12 simplices:
\[(0,1,2,5,8),(0,1,3,5,9),(0,3,4,5,9),(0,2,4,5,9),(0,1,5,9,10),(0,1,5,8,10),\]
\[(0,2,5,9,10),(0,2,5,8,10),(0,2,4,9,10),(0,2,4,7,10),(0,1,3,9,10),(0,1,3,6,10),\]
where each simplex is represented by the vertices it contains with the numbering of vertices given above (according to the monomial order), starting from $0$.

The polytope can be considered in the lattice $\Z^4$ with the standard basis. Let $(e_i)$ be the dual basis.
The fan of the polytope is given by $9$ rays: $\pm e_i$ for $1\leq i\leq 4$, and $e_1+e_2+e_3+e_4$. It has $11$ maximal cones, corresponding to vertices. It is \emph{not} the fan obtained by the blow up of $(\P^1)^4$ in one point. In particular, it has $6$ maximal cones that are not simplicial. This should not be surprising, as the singular locus of the secant variety strictly contains the Segre variety. This implies that the $0$ point of the affine toric variety we consider should not be the only singular point.
\end{exmp}

Let us prove further interesting properties of the toric varieties that have appeared in our construction. We will prove the existence of degree two square-free Gr\"obner basis. We adapt the methods from \cite[Section 14A]{Sturmfels96}. First let us introduce new notation.

%To each interval $[k_i]$ we associate the set $A_i:=\{a_i^1,\dots,a_i^{k_i}\}$. The integer $l\in [k_i]$ corresponds to the element $a_i^l$ if $l\neq 0$. None of the elements corresponds to $0\in [k_i]$. To a sequence $I=(i_1,\dots,i_n)$, where $i_j\in \{0,\ldots,k_j\}$ we associate the \emph{set} $\{a_1^{i_1},\dots,a_n^{i_n}\}$, where we omit $a_j^{i_j}$ for $i_j=0$.
%\piotr{First, note that we defined $[k_i]$ as $\{1,\ldots,k_i\}$. Second, I would completely rewrite everything beginning from "To each interval..." and replace it with one sentence like:

%Consider a set of formal symbols $A_i:=\{a^i_1,\dots,a^i_{k_i}\}$ for each $i=1,\ldots,n$ and a map
%$$
%I=(i_1,\ldots,i_n)\quad\mapsto \quad \{a^j_{i_j}:\,i_j\neq 0\}.
%$$

%}

%In particular, it is a bijection between arbitrary sequences $I$ and such subsets of $\bigcup A_i$, that intersect each $A_i$ in at most one element. Recall that the variables $z$ were indexed by sequences $I=(i_1,\dots,i_n)$ such that $|I|\geq 2$. This motivates the following definition.
\begin{defn}[$A_I$, $\bar{A}_I$, admissible]
To each $I=(i_1,\ldots,i_n)$ we associate a set of pairs
$$
A_I\quad:=\quad\{(j,i_j):\, j=1,\ldots,n \mbox{ such that } i_j\neq 0\}.
$$
Such pairs are naturally ordered by
$$
(1,1)<\cdots<(1,k_1)<(2,1)<\cdots<(n,k_n).
$$
We add a pair $(\infty,\infty)$ to $A_I$ with multiplicity $n-|I|$ to obtain a multiset $\bar{A}_I$ with exactly $n$ elements.
%Let $A_i:=\{a^i_1,\dots,a^i_{k_i}\}$ be sets where $a^i_j$ are formal symbols.
We call a set of $A_I$ \emph{admissible} if
%\begin{enumerate}
%\item
it is of cardinality at least $2$.
%\item its intersection with each $A_i$ is of cardinality at most $1$.
%\end{enumerate}
%We also introduce a formal symbol $\infty$.
We call a multiset $\bar{A}_I$ \emph{admissible} if $A_I$ is.
%of
%$\{\infty\}\cup\bigcup_{i=1}^n A_i$ admissible if it is of cardinality $n$ and its intersection with $\bigcup_{i=1}^n A_i$ is an admissible subset. Less formally admissible multisubsets are in bijection with admissible subsets by adding/removing the element $\infty$ sufficiently many times.
\end{defn}
%\piotr{To me it seems more natural to generalize Bernd's approach without introducing $a_i^j$. More like that:

%To each $I=(i_1,\ldots,i_n)$ we associate a set of pairs
%$$
%A_I\quad:=\quad\{(j,i_j):\, j=1,\ldots,n \mbox{ such that } i_j\neq 0\}.
%$$
%Such pairs are naturally ordered by
%$$
%(1,1)<\cdots<(1,k_1)<(2,1)<\cdots<(n,k_n).
%$$
%We add a pair $(0,0)$ to $A_I$ with multiplicity $n-|I|$ to obtain a multiset $\bar{A}_I$ with exactly $n$ elements.
%(admissibility has a natural definition etc)
%}
%\mateusz{For me it is completely fine. In fact it really does not change much - one represents $a_i^j$ as a pair $(i,j)$. I would prefer to add $(\infty,\infty)$ so that there are not so many changes, instead of $(0,0)$ but of course we can change to $(0,0)$.}\piotr{Of course I realized that this is just mapping $a_j^i$ to $(i,j)$. I just mean that things may be easier to see with less formalities :p}

Recall that the variables $z$ were indexed by sequences $I=(i_1,\dots,i_n)$ such that $|I|\geq 2$.
W can index the variables $z$ by admissible sets or equivalently admissible multisets.
%Let us order the set $\bigcup A_i$ by setting $a^{i_1}_{j_1}<a^{i_2}_{j_2}$ if $i_1<i_2$ or $i_1=i_2$ and $j_1<j_2$. We also set $a^i_j<\infty$.

Consider a pair $(C,D)$ of admissible multisets, where $C=\{c_1,\dots,c_n\}$, $D=\{d_1,\dots,d_n\}$. We call the pair $(C,D)$ \emph{sorted} if $c_1\leq d_1\leq c_2\leq d_2\leq\dots\leq c_n\leq d_n$.

Consider any quadratic monomial $z_Az_B$, where $A,B$ are admissible multisets. Define multiset $S:=A\cup B$. We claim that there exists a unique sorted pair $(C,D)$ of admissible multisets, such that $S=C\cup D$. Indeed, it is enough to sort elements of $S$ and define $C$ as elements appearing on odd places and $D$ as elements appearing on even places. Notice that $C$ and $D$ are admissible. %Indeed, the intersection of $A_i$ with $S$ is of cardinality at most $2$, and when we sort $S$ such elements must be next to each other. Hence the intersection of $C$ and $D$ with $A_i$ is of cardinality at most $1$. Moreover the cardinalities of the intersection of $C$ or $D$ with $\bigcup A_i$ can differ at most by one. As the cardinality of $S\cap\bigcup A_i$ is at least $4$ we see that $C$ and $D$ are admissible.
Thus the quadratic binomial $z_Az_B-z_Cz_D$ belongs to the ideal of the toric variety associated to the polytope $P$. Moreover if $A=B$ then $A=B=C=D$ and the binomial is zero, so we may assume that $z_A\neq z_B$.
\begin{defn}[Basis binomials, leading terms]\label{def:basisbinomleadmon}
Consider any two admissible multisets $A,B$, such that the pairs $(A,B)$ and $(B,A)$ are not sorted. Let $(C,D)$ be a sorted pair constructed above. We call the quadratic binomial $z_Az_B-z_Cz_D$ a \emph{basis binomial}. We also \emph{define} $z_Az_B$ to be its \emph{leading term}. By the following theorem there really exists a term order with respect to which $z_Az_B$ is a leading term.%\piotr{I don't understand this last statement. Especially given the next theorem. Maybe replace it with: "By the following theorem there really exists a term order with respect to which $z_Az_B$ is a leading term."}
\end{defn}
\begin{thm}\label{thm:deg2GB}
There exists a term order, with respect to which the basis binomials form a Gr\"obner basis. Moreover, with respect to this term the leading monomials are as defined in~\ref{def:basisbinomleadmon}. In particular, the Gr\"obner basis is quadratic and the leading monomials are square-free.
\end{thm}
\dow
The main concept of the proof is to apply \cite[Theorem 3.12]{Sturmfels96}. It is enough to prove that the reduction using basis binomials with marked leading terms in noetherian. In other words we have to prove that there is no infinite sequence of monomials $m_1\rightarrow m_2\rightarrow\dots$ where $m_{i+1}$ is obtained from $m_i$ by replacing two variables $z_Az_B$ by $z_Cz_D$, where $A,B,C,D$ are as in Definition~\ref{def:basisbinomleadmon}. To show this, to each monomial $m$ we associate a nonnegative integer $l(m)$ and we show that $l(m_{i+1})<l(m_i)$. Fix a monomial $m=z_{B_1}\cdots z_{B_d}$ of degree $d$, where $B_i=\{b^i_1,\dots,b^i_n\}$ are admissible multisets -- recall that each $b^i_j$ is a pair. To fix the notation we assume that $b^i_j\leq b^i_{j+1}$. Consider any permutation $\sigma$ of the set $\{1,\dots,d\}$. To this permutation we associate a sequence:
\[(b^{\sigma(1)}_1,b^{\sigma(2)}_1,\dots,b^{\sigma(d)}_1,b^{\sigma(1)}_2,\dots,b^{\sigma(d)}_2,\dots,b^{\sigma(1)}_n,\dots,b^{\sigma(d)}_n).\]
We define $l_\sigma(m)$ as the number of inversions in the above sequence, that is a number of pairs of indices $i<j$ such that the $i$-th element of the sequence is greater than the $j$-th. We also define $l(m)=\min_\sigma l_\sigma(m)$. It is an obvious observation that each reduction with basis binomials strictly decreases $l$.
\kdow
The following is the main result of this section.
\begin{thm}\label{cor:CM2}
The secant variety ${\rm Sec}(\P(V_{1})\times \dots \times \P(V_{n}))$ is covered by normal toric varieties. Hence it has rational singularities and in particular it is normal and Cohen-Macaulay.\kwadrat
\end{thm}
\begin{rem}
Some special cases of Theorem~\ref{cor:CM2} were known. For the variety $\P^1\times\P^a\times\P^b$ the statement about normality and rational singularities can be found in \cite[Theorem~1.1]{LandsbergWeyman}.
For arbitrary three factors the fact that the singular locus has codimension $2$ is stated in \cite[Theorem~1.3]{BuczynskiLandsberg}. The statement about the normality, in the general case, can be found in \cite[Theorem~2.2]{Vermeire}.
\end{rem}

\subsection{Singular locus of the secant}
Let us describe the singular locus of the secant variety of any Segre. First let us describe the polytope $P$ and the fan of the associated toric variety.
\begin{defn}[Projection $\pi$, $Q$]
Let $\pi$ be the projection of lattices $\pi:\Z^{k_1+1}\times\dots\times\Z^{k_n+1}\rightarrow  \Z^{k_1}\times\dots\times\Z^{k_n}$ defined by forgetting in each component the first coordinate. We also define $Q:=\pi(P)$. Note that $P$ and $Q$ are isomorphic as polytopes, but the monoids they span are different.
\end{defn}
Let us provide a precise facet description of the polytope $Q$.
\begin{prop}\label{prop:facetdesc}
The polytope $Q$ is given by points $q = (q_{i}^{j_{i}})$  satisfying the inequalities:
\begin{itemize}
\item $q^i_j\geq 0$\qquad\mbox{for all $i=1,\ldots,n$ and $j=1,\ldots,k_i$};
\item $\sum_{j=1}^{k_i}q^i_j\leq 1$\qquad\mbox{for all $i=1,\ldots,n$};
\item $\sum_{i=1}^n\sum_{j=1}^{k_i}q^i_j\geq 2$.
\end{itemize}
\end{prop}
\dow
Let $Q'$ be the polytope given by the inequalities.
It is straightforward that points of $Q$ lie in  $Q'$. We will prove inductively on $d$ that each \emph{integer} point of $dQ'$ belongs to $dQ$. This will prove that all rational points of $Q'$ belong to $Q$ what proves the other inclusion.
For $d=1$ this is a direct check. Suppose the statement is true for $d$. Let $(a^i_j)\in (d+1)Q'$ be an integral point. Let us consider three subcases.
\begin{itemize}
\item There exists $i_1\neq i_2$ such that $\sum_{j=1}^{k_{i_1}}a^{i_1}_j=\sum_{j=1}^{k_{i_2}}a^{i_2}_j=d+1$.

    In this case let $I$ be the set of all indices with the above property. By assumption we have $|I|\geq 2$. For each $i\in I$ choose $a^i_{j_i}\neq 0$. Let us write $(a^i_j)=(b^i_j)+(c^i_j)$, where $b^i_j:=a^i_j$ if $i\not\in I$ or $i\in I$ and $j\neq j_i$, $b^i_{j_i}=a^i_{j_i}-1$ for $i\in I$ and $c^i_j=0$ unless $i\in I$ and $j=i_j$ in which case $c^i_{j_i}=1$. If is straightforward that $(b^i_j)\in dQ'=dQ$ and $(c^i_j)\in Q$, hence $(a^i_j)\in (d+1)Q$.
\item There exists exactly one $i_0$ such that $\sum_{j=1}^{k_{i_0}}a^{i_0}_j=d+1$.

As $\sum_{i=1}^n\sum_{j=1}^{k_i} a^i_j\geq 2d+2$ we may choose $a^{i'}_{j'}\neq 0$ for $i'\neq i_0$. We also choose $a^{i_0}_{j_0}\neq 0$. Defining $(a^i_j)=(b^i_j)+(c^i_j)$ with $c^{i'}_{j'}=c^{i_0}_{j_0}=1$ and $c^i_j=0$ otherwise we may conclude as in the previous example.

\item There are no $i$ such that $\sum_{j=1}^{k_{i_0}}a^{i_0}_j=d+1$.

In this case we just choose any two nonzero $a^i_j$ for different $i$, what is possible as $\sum_{i=1}^n\sum_{j=1}^{k_i} a^i_j\geq 2d+2$, and conclude as before.\qedhere
\end{itemize}
\kdow
\begin{rem}
Note that the proof of Proposition~\ref{prop:facetdesc} proves normality of the polytope $Q$ and hence of $P$. Indeed, we have shown that each lattice point of $(d+1)Q$ is a sum of two lattice points respectively from $Q$ and $dQ$. Thus, by induction on $d$, each lattice point of $dQ$ is a sum of $d$ lattice points from $Q$.
\end{rem}

Consider the lattice $\Z^{k_1}\times\dots\times\Z^{k_n}$ with the standard basis. Let $e_j^i$ be the dual basis. The three types of facets of Proposition~\ref{prop:facetdesc} correspond to three types of ray generators of the dual fan:
 \begin{enumerate}
\item $R_i^j=e^i_j$;
\item $L_i=\sum_{j=1}^{j=k_i}-e^i_j$;
\item $S=\sum_{i,j}e^i_j$.
\end{enumerate}
\begin{lem}\label{lem:notcone}
For a fixed $i$, the intersection of the facet corresponding to $L_i$ with all facets corresponding to $R_i^j$ is empty. This implies that for any $i$ the rays $R_i^j$ and $L_i$ do not form a cone for any $j$.\kwadrat
\end{lem}
\begin{prop}\label{prop:singloc}
The singular locus of the projective variety associated to $Q$ has $n\choose 2$ components indexed by two element subsets of the set $[n]$.
\end{prop}
\dow
From Lemma~\ref{lem:notcone} it follows that cones that do not contain the ray $S$ are smooth. Consider the cone that contains $S$. Choose any two rays $L_{i_1}, L_{i_2}$. Note that  $S, L_{i_1}, L_{i_2}$ form a cone, as the intersection of the corresponding facets is nonempty. Moreover, each point of the intersection belongs to any facet corresponding to $R_i^j$ for any $i\neq i_1,i_2$. Thus any cone containing $S, L_{i_1}, L_{i_2}$ must also contain all $R_i^j$ for any $i\neq i_1,i_2$. As they are linearly dependent, the cone is not simplicial, hence not smooth. Thus we see that cones containing $S$ and any two $L_{i_1}, L_{i_2}$ are not smooth. As the intersection of the facets corresponding to $S$ and three different facets corresponding to $L_i$ is empty, all these cones are different.

We will now prove that if a cone containing $S$ does not contain any pair of $L_{i_1}, L_{i_2}$ then it is smooth. Suppose it does not contain any rays $L_i$. Then the cone could be not smooth only if it contained all the rays $R_i^j$. This is of course impossible, as the facets corresponding to all rays $R_i^j$ have an empty intersection.

Suppose the cone contains exactly one ray $L_{i_0}$. Then, by Lemma~\ref{lem:notcone} it does not contain all rays $R_{i_0}^j$. Thus it would not be smooth only if it contained all the rays $R_i^j$ for $i\neq i_0$. These do not form a cone, which finishes the proof.
\kdow
As each component of the singular locus contains the Segre we obtain the following Corollary.
\begin{cor}\label{cor:components}
The singular locus of the secant of the Segre variety has $n\choose 2$ components indexed by two-element subsets of the set $[n]$. The component indexed by the set ${i_1,i_2}$ has codimension $\sum_{i\neq i_1,i_2} k_i -2$. The intersection of any two components is the Segre variety.

Specifically, each component of the singular locus of the secant of the Segre variety is isomorphic to:
\[
\P^{k_{1}}\times\dots\times \widehat{\P^{k_{i_{1}}}} \times \dots \times \widehat{\P^{k_{i_{2}}}} \times\dots \times \P^{k_{n}} \times {\rm Sec}(\P^{k_{i_{1}}}\times\P^{k_{i_{2}}}),\]
where $\widehat{\cdot}$ denotes omission.
\end{cor}
\dow
The statement about the intersection of the components follows from the fact that the sum of any two cones described in the proof of Proposition~\ref{prop:singloc} does not form a cone.

As before we can parameterize each component of the singular locus by restricting the usual parameterization to $a^i_j=b^i_j$ for $i$ not in the chosen subset.
In particular, as in Section~\ref{sec:pultriangsingloc}, we see that the components of the singular locus are as described in the theorem.
\kdow
%\piotr{I think that the following result should be called a theorem.}
%\mateusz{We should probably say that smoothness results are uninteresting. It is obvious that (for any variety $X$) the tangent space of the secant at a point belonging to the variety is the whole space (if $X$ is not contained in any projective subspace), so the smooth cases are exactly those when the secant fills the ambient.}
\begin{thm}\label{thm:secGor}
Consider the secant of the Segre $\P^{k_1}\times\dots\times\P^{k_n}$. Assume that $k_1\leq k_2\leq\dots \leq k_n$.
The secant is smooth only in one of the following cases:
\begin{enumerate}
\item $n=3$ and $k_1=k_2=k_3=1$,
\item $n=2$ and $k_1=1$.
\end{enumerate}
If the secant is not smooth it is not $\Q$-factorial. Moreover, if it is not smooth it is Gorenstein only in
%\luke{I think ``one of'' should be deleted here}
 the following cases:
\begin{enumerate}
\item $n=5$ and $k_1=k_2=k_3=k_4=k_5=1$,
\item $n=3$ and $(k_{1},k_{2},k_{3})$ equal to one of $(1,1,3)$, $(1,3,3)$, or $(3,3,3)$.
\item $n=2$ and $k_1=k_2$.
\end{enumerate}
In the above cases the secant has terminal singularities. If the secant is not Gorenstein, then it is not $\Q$-Gorenstein.
\end{thm}
%\mateusz{So I suggest the following:}\luke{I agree - nice comment!}
\begin{rem}
Note that the tangent space to the secant ${\rm Sec}(X)$ at the point belonging to $X$ equals the whole linear space spanned by $X$. In particular the secant can be smooth if and only if it fills the whole ambient space. The results when the secant of the Segre embedding fills the ambient space are quite standard, thus the interesting classification of the previous theorem concerns Gorenstein cases.
\end{rem}

\dow[Proof of Theorem~\ref{thm:secGor}]
It is easy to check that given cases are smooth. Let us prove that all the other secants are not $\Q$-factorial. Let $\sigma$ be the cone over the polytope $P$, corresponding to the toric variety that is open in the secant. By Proposition~\ref{prop:facetdesc} the cone $\sigma$ is defined by the inequalities:
\begin{enumerate}
\item $q^i_j\geq 0$\qquad\mbox{for all $i=1,\ldots,n$ and $j=1,\ldots,k_i$};
\item $\sum_{j=1}^{k_i}q^i_j\leq q_0$\qquad\mbox{for all $i=1,\ldots,n$};
\item $\sum_{i=1}^n\sum_{j=1}^{k_i}q^i_j\geq 2q_0$.
\end{enumerate}
If $n\geq 4$ then all the inequalities are supporting. In particular the dual cone is not simplicial, thus the variety is not $\Q$-factorial. The same holds if $n=3$ and $k_1\geq 2$. Assume $n=3$, $k_1=1$ and $k_3\geq 2$. Among the supporting hyperplanes we have:
\begin{enumerate}
\item 1 hyperplane $-2q_0+\sum_{i,j}q^i_j=0$,
\item 3 hyperplanes $q_0-\sum_j q^i_j=0$,
\item $k_3$ hyperplanes $q^3_j=0$.
\end{enumerate}
In particular the dual cone is not simplicial. Let us now assume $n=2$. The polytope $P$ is the product of simplices of dimension $k_1-1$ and $k_2-1$. In particular it is of dimension $k_1+k_2-2$ and, if $k_1>1$ it has $k_1+k_2$ supporting facets. Thus, if $k_1>1$ the dual of $\sigma$ is not simplicial. Notice that in this case although the polytope $P$ is smooth, that is defines a smooth \emph{projective} variety - product of projective spaces under Segre embedding, the \emph{affine} cone over this variety is \emph{not} smooth.

Now we classify secants that are $\Q$-Gorenstein. First let us consider the case when $n\geq 4$ and $k_i$ are arbitrary or $n=3$ and $k_1>1$. In this case all the given inequalities are supporting for $\sigma$. Thus the rays of the dual cone are:
\begin{itemize}
\item $e^i_j$;
\item $e_0-\sum_j e^i_j$;
\item $-2e_0+\sum_{i,j}e^i_j$.
\end{itemize}
Suppose they belong to the affine subspace given by
\[\lambda_0e_0^*+\sum_{i,j}\lambda^i_je_j^{i*}=a,\]
for some constant $a$. We must have $\lambda^i_j=a$, $\lambda_0-k_ia=a$ and $-2\lambda_0+\sum_i ak_i=a$. In particular we must have $k_1=\dots=k_n$ and hence $(n-2)k_1=3$. This has solutions only if $n=3$ and $k_1=3$ or $n=5$ and $k_1=1$. In these two cases one can put $a:=1$, $\lambda^i_j=1$ and $\lambda_0:=k+1$, proving that the variety is Gorenstein.

We now check the case when not all inequities given above are supporting $\sigma$. Suppose now $n=3$ and $k_1=1$. Assume first $k_2\geq 2$. The rays of the dual cone $\sigma^\vee$ are as given above, omitting $e_1^1$. Assuming, as previously, that they belong to a hyperplane we obtain:
$\lambda^i_j=a$ and $\lambda_0-ak_i=a$ for for $i\neq 1$. Moreover $\lambda_0-\lambda_1^1=a$ and $-2\lambda_0+\lambda_1^1+ak_2+ak_3=a$. This has a solution only if $k_2=k_3=3$ and then, the solution is integral if we put $a=1$. The case $k_2=1$, $k_3\geq 2$ is analogous.

Suppose $n=2$. We may assume $k_1>1$.  The polytope $P$ is the product of two simplices of dimensions $k_1-1$ and $k_2-1$. Thus, in appropriate coordinates the rays of the dual cone $\sigma^*$ are:
\begin{itemize}
\item $k_1-1$ rays $e_i$,
\item $k_2-1$ rays $f_j$,
\item one ray $e_0-\sum e_i$,
\item one ray $e_0-\sum f_j$.
\end{itemize}
Suppose they belong to a hyperplane $\lambda_0e_0^*+\sum\lambda_ie_i^*+\sum\lambda_j'f_j=a$. We obtain $\lambda_i=\lambda_j'=a$, $\lambda_0-a(k_1-1)=a$ and $\lambda_0-a(k_2-1)=a$. These have a solution if and only if $k_1=k_2$. Moreover in this case, for $a:=1$ one finds an integral solution.

It is a direct check that all the Gorenstein cases have terminal singularities, by proving that there are no integral points in the convex hull of ray generators of $\sigma^\vee$ and $0$.
\kdow
\section{The tangential variety}\label{sec:tangent}

In this section we show that our techniques can be applied to the study of the tangential variety of the Segre.

\subsection{Local structure of the tangential variety}
On the open set $U_{{(0,\dots,0)}}$, the tangential variety ${\rm Tan}(\P(V_{1})\times \dots \times \P(V_{n}))$ consists of all points on all lines
%on limits to honest \piotr{what does "honest" mean here? this sounds colloquial}\mateusz{I do not remember writing it. I think Luke ment "honest" as lines intersecting a variety in two points. For me tangent lines are much more natural then limits of honest secants.}\luke{This is a term that JM uses to mean those secants that are actually secant (i.e. in two points)}
tangent to the Segre. It can be parameterized by
\begin{multline}\label{eq:tanpar}
\left([a_{0}^{1},\ldots,a^{1}_{k_1}],\dots,[a^{n}_{0},\ldots,a^{n}_{k_n}] \right),
\left([b_{0}^{1},\ldots,b^{1}_{k_1}],\dots,[b^{n}_{0},\ldots,b^{n}_{k_n}] \right)\longmapsto\\
x_{(i_1,\dots,i_n)}=1/n \sum_{k=1}^{n}\left( b_{i_{k}}^{k}\prod_{j\in[n]\setminus\{k\}}a^j_{i_j}\right)
,
\end{multline}
for all $i_j\in \{0,\ldots,k_j\}$, where $a_{i_{j}}^{j}$ and $b_{i_{j}}^{j}$ are $\C$ valued parameters.

We introduce the affine variety $W$ given by
\[
W\quad:=\quad {\rm Tan}(\P(V_{1})\times \dots \times \P(V_{n})) \cap U_{{(0,\dots,0)}}.
\]
As before we may restrict the parameterization of $W$ by setting $a_{0}^{i}=b_{0}^{i}=1$.
\begin{prop}
The variety $W$
is isomorphic to the trivial affine bundle of rank $\sum_{i=1}^n k_i$ over the spectrum of the semigroup algebra associated to the monoid generated by $Q$. \end{prop}
\begin{proof}Given the parameterization of the variety $W$ we proceed along similar lines as in Lemma~\ref{lem:parametryzacja} to show that the induced parameterization in $(y_I)$ is monomial. More precisely, for $|I|\geq 2$ a point of the tangential parameterized by some $a^k_{i_k}=b^k_{i_k}$ for $i_k\in Supp(I)$ satisfies $y_I=0$. Thus, for $|I|\geq 2$ the parameterization of the tangential is given by $y_I=\gamma \prod_{i_k\in Supp(I)}(b^k_{i_k}-a^k_{i_k})$ for some constant $\gamma$. In particular, the parameterization is monomial in parameters $(b^k_{i_k}-a^k_{i_k})$.
\end{proof}
\begin{rem}
Sturmfels and Zwiernik proved that the tangential variety has a monomial parameterization in cumulant coordinates,  \cite{SturmfelsZwiernik}[Thm.~4.1].  Here we have shown that this occurs already in central moment coordinates, which implies their result.
\end{rem}

\begin{prop}
The complement of a hyperplane section of the tangential variety to the Segre $\P(V_1)\times\dots\times \P(V_n)$
is isomorphic to the trivial affine bundle over the spectrum of the semigroup algebra associated to the monoid generated by $Q=\pi(P)\subset \Z^{k_1}\times\dots\times\Z^{k_n}$. \kwadrat
\end{prop}
In other words, we look at the same points that defined the secant, but we forget the first coordinate. Note that we do not get a homogeneous ideal any more. To describe the tangential variety locally we need a description of the monoid generated by $Q$.
\begin{lem}\label{lem:monoidofQ}
The monoid spanned by lattice points in $Q$ is the intersection of the lattice with the cone defined by the points $q = (q^{i}_{j_{i}}) \in \Z^{k_1}\times\dots\times\Z^{k_n} $, where $1\leq i \leq n$,  $1\leq j_{i}\leq k_{i}$,  satisfying the following inequalities:
\[
q^i_j\geq 0,\qquad
\sum_{j=1}^{k_i}q^i_j\leq \sum_{l\neq i, 1\leq i\leq n}^n\sum_{s=1}^{k_l} q^l_s.
\]
If $n\geq 4$ and each $k_{i}\geq 1$ or $n=3$ and each $k_i\geq 2$ then all of the given inequalities are supporting for the cone.
\end{lem}
\dow
It is obvious that all the points of the monoid satisfy these inequalities, as all the points of $Q$ do.
Let us prove the other inclusion, inductively on the sum of coordinates.
Consider any integral point $(a_i^j)$ that satisfies the inequalities. Without loss of generality suppose
\[\sum_{j=1}^{k_1}a^1_j\geq\dots\geq \sum_{j=1}^{k_n}a^n_j.\]
If the point is nonzero we must have $\sum_{j=1}^{k_1}a^1_j,\sum_{j=1}^{k_2}a^2_j\geq 1$. Suppose first that $\sum_{j=1}^{k_1}a^1_j=\sum_{j=1}^{k_2}a^2_j=\sum_{j=1}^{k_3}a^3_j=1$ and $\sum_{j=1}^{k_4}a^4_j=\dots=\sum_{j=1}^{k_n}a^n_j=0$. In this case the point lies in $Q$. Suppose now that this is not the case. To fix the notation suppose that $a_1^1,a^2_1\geq 1$. Consider the point $(b^i_j)$, where $b^1_1=a_1^1-1$, $b^2_1=a^2_1-1$ and $b^i_j=a^i_j$ otherwise. We claim that it also satisfies the inequalities. Indeed the only nontrivial thing we have to check is
\[\sum_{j=1}^{k_3}a^3_j\quad\leq\quad \sum_{j=1}^{k_1}a^1_j+\sum_{j=1}^{k_2}a^2_j+\sum_{j=1}^{k_4}a^4_j+\dots+\sum_{j=1}^{k_n}a^n_j-2.\]
 As $\sum_{j=1}^{k_3}a^3_j\leq \sum_{j=1}^{k_2}a^2_j$ this could not hold only if $\sum_{j=1}^{k_1}a^1_j=1$, $\sum_{j=1}^{k_3}a^3_j=\sum_{j=1}^{k_2}a^2_j$ and $\sum_{j=1}^{k_4}a^4_j=\dots=\sum_{j=1}^{k_n}a^n_j=0$. As the point satisfies the inequalities, we would have $\sum_{j=1}^{k_1}a^1_j=\sum_{j=1}^{k_3}a^3_j=\sum_{j=1}^{k_2}a^2_j=1$ and $\sum_{j=1}^{k_4}a^4_j=\dots=\sum_{j=1}^{k_n}a^n_j=0$.
This case was excluded before. By induction we finish the proof of the description of the cone.

The last statement is a direct check.
\kdow

The following result is a direct consequence of Lemma~\ref{lem:monoidofQ}.
\begin{prop}
The tangential variety of the Segre is a normal variety with rational singularities. In particular it is Cohen-Macaulay.\kwadrat
\end{prop}

We now provide a detailed description of the singular locus of the tangential variety.

\subsection{The singular locus}

We can also provide rays of the cone generated by $Q$, which will provide a better understanding of the singular locus of the secant.
\begin{prop}
The ray generators of the cone spanned by $Q$ are precisely those integral points of $Q$ that have exactly two coordinates equal to $1$.
\end{prop}
\dow
Let $R$ be the set of the integral points of $Q$ with exactly two nonzero coordinates.
It is an easy observation that all other integral points of $Q$ are in the $\R_+$ span of $R$. Moreover each of the elements of $R$ is not in the $\R_+$ span of the others. As the points from $R$ are primitive the proposition follows.
\kdow

Following the standard toric analysis we now describe the dual cone of the cone generated by $Q$. First, recall a few facts. The faces of the dual cone are in bijection with the faces of the cone. In particular rays of the dual cone correspond to facets of the cone. If we take a face $F$ of the cone, the corresponding face of the dual cone is spanned by exactly those rays that correspond to facets containing $F$.

In the cases $n\geq 4$ and each $k_{i}\geq 1$ or $n=3$ and each $k_i\geq 2$ there are two types of facets, described by points
$q = (q^{i}_{j_{i}}) \in \Z^{k_1}\times\dots\times\Z^{k_n}$  satisfying one of the following:
\begin{enumerate}
\item ``forcing zero facets'' $q^i_j=0$, indexed by $i,j$;
\item ``forcing one facets'' $\sum_{j=1}^{k_i}q^i_j= \sum_{l\neq i, 1\leq i\leq n}^n\sum_{s=1}^{k_l} q^l_s$, indexed by $i$.
\end{enumerate}
The name of the first facets is self-explanatory. The motivation for the second name is as follows. If we consider the ray generators that belong to a given ``forcing one facet'' indexed by $i$ we see that these are exactly those that have got one of the coordinates $q^i_j$ equal to $1$ for some $j$.
\begin{lem}\label{lem:Qrays}
Suppose $n\geq 4$ and each $k_{i}\geq 1$ or $n=3$ and each $k_i\geq 2$.
Let $C$ be the dual cone of the cone spanned by $Q$ or equivalently by $R$. If a subcone of $C$ contains two rays that correspond to ``forcing one facets'' indexed by $i_1$ and $i_2$, then it also contains all rays corresponding to ``forcing zero facets'' indexed by any $i\neq i_1,i_2$ and any $j$. Such subcones are not smooth. All other subcones are smooth.
\end{lem}
\dow
Consider two ``forcing one facets'' indexed by $i_1$ and $i_2$. The elements from $R$ that belong to their intersection must have $1$ for some coordinates $q^{i_1}_{j_1}$ and $q^{i_2}_{j_2}$. As they have exactly two ones all other $q^i_j$ must be equal to zero. This proves the first part of the lemma.

The sum of facet normals of two chosen ``forcing one facets'' equals twice the sum of all ``forcing zero facets'' for $i\neq i_1,i_2$. In particular they are linearly dependent. Hence the cone is not simplicial, thus not smooth.

We only have to prove that all the other cones are smooth. If a cone contains only rays corresponding to ``forcing zero facets'' then its rays form a part of the basis, thus it is smooth. Suppose it contains only one ``forcing one facet''. The corresponding ray is of the type $(\pm 1)$. If the cone contains not all rays corresponding to ``forcing zero facets'' it is smooth, as they correspond to standard basis vectors. On the other had a cone, not equal to $C$, cannot contain all ``forcing zero facets'', as there are no elements of $R$ contained in all ``forcing zero facets''.
\kdow
%\piotr{Why there is a clear page command below in the code?}
% \clearpage

The following is a direct consequence of Lemma~\ref{lem:Qrays}.
\begin{cor}
Consider the tangential variety of the Segre embedding of $\P^{k_1}\times\dots\times \P^{k_n}$, where $k_1\leq\dots \leq k_n$ and $n\geq 3$.
The singular locus has $n\choose 2$ components, corresponding to two element subsets of $\{1,\dots,n\}$. The codimension of the component corresponding to the subset $\{i,j\}$ equal $1+\sum_{l\neq i,j} k_l$.
Specifically, the component of the singular locus corresponding to the subset $\{i,j\}$ is isomprphic to
\[{\rm Tan}(\P^{k_{i}}\times{\P^{k_{j}}}) \times \P^{k_{1}}\times\dots\times \widehat{\P^{k_{i}}} \times \dots \times \widehat{\P^{k_{j}}} \times\dots \times \P^{k_{n}}.\]
\end{cor}
\dow
The only case not covered so far is $n=3$. This is analogous to other cases.
Let us consider the parameterization of the tangential variety in \eqref{eq:tanpar}. Choose two indices $i,j$. Restrict the parameterization to $a^{l}_{k}=b^{l}_{k}$ for $l\neq i,j$. We can see that the image of the parameterization map is contained in the singular component corresponding to $\{i,j\}$. By the dimension count they must be equal. The corollary follows.
\kdow
For $n=2$ it is straightforward to check that the tangential variety equals the secant
%\mateusz{we can give a reference to Zak, prove it directly, give some other reference, or just leave it as obvious}.
In this case we obtain a classification of singularities by Theorem~\ref{thm:secGor}. In other cases we present the following theorem.
\begin{thm}
Assume $n\geq 3$. The tangential variety is always singular. There is only one $\Q$-Gorenstein tangential variety given by $n=3$ and $k_1=k_2=k_3=1$. In this case the tangential variety is $\Q$-factorial, Gorenstein and terminal.
\end{thm}
\dow
Let $\sigma$ be the cone spanned by $Q$.
First consider the case $n\geq 4$ and $k_i$ arbitrary or $n=3$ and $k_i\geq 2$. By Lemma~\ref{lem:monoidofQ} the dual cone of $\sigma$ has got the following ray generators:
\begin{enumerate}
\item $e^i_j$;
\item $\sum_{i\neq i_0}\sum_je_j^i-\sum_j e^{i_0}_j$.
\end{enumerate}
This cone is not simplicial, so the variety is not $\Q$-factorial. It is a direct check that the ray generators do not belong to a hyperplane, so the variety is not $\Q$-Gorenstein.

Assume now $n=3$ and $k_1=1$. If $k_2=k_3=1$ the dual cone of $\sigma$ has three ray generators:
$$e_1+e_2-e_3,e_1+e_3-e_2,e_2+e_3-e_1.$$
Thus the dual cone is simplicial and the ray generators belong to the hyperplane $e_1^*+e_2^*+e_3^*=1$. It is not smooth, as the sum of all ray generators is not a primitive lattice element. However, one can see that the dual cone is terminal.

Without loss we can assume $k_3>1$. We will show that the tangential variety is not $\Q$-Gorenstein. First consider the case $k_2=1$. The ray generators of the dual cone are:
$$e^3_j, e_2+\sum e^3_j-e_1, e_1+\sum e^3_j-e_2,e_1+e_2-\sum e^3_j.$$
These do not belong to an affine subspace. The case $k_2>1$ is analogous.
\kdow
\section*{Acknowledgements}
The authors would like to thank
Weronika Buczy{\'n}ska, Jaros\l aw Buczy{\'n}ski, Jan Draisma, Christian Haase, Nathan Ilten, Diane Maclagan, Claudiu Raicu, Steven Sam, and Bernd Sturmfels
for useful discussions, references, and comments.

%\bibliographystyle{siam}
%\bibliography{../!bibliografie/algebraic_statistics}
\end{document}